\documentclass[review,3p]{elsarticle}

\usepackage{lineno,hyperref}
\usepackage{amsmath,amsfonts,amssymb,amsthm}
\usepackage[utf8x]{inputenc}
\usepackage[dvipsnames,table]{xcolor}
\usepackage{algorithm}
\usepackage{algorithmic}
\usepackage{blindtext}
\usepackage{lipsum}
\usepackage{graphicx}
\usepackage{caption}
\usepackage{subcaption}
\biboptions{longnamesfirst,sort&compress}

\usepackage{scalerel,stackengine}
\stackMath
\newcommand\reallywidehat[1]{%
\savestack{\tmpbox}{\stretchto{%
  \scaleto{%
    \scalerel*[\widthof{\ensuremath{#1}}]{\kern-.6pt\bigwedge\kern-.6pt}%
    {\rule[-\textheight/2]{1ex}{\textheight}}
  }{\textheight}%
}{0.5ex}}%
\stackon[1pt]{#1}{\tmpbox}%
}

\usepackage{tikz}

\allowdisplaybreaks

\journal{Applied Mathematics and Computation}

\theoremstyle{theorem}

\numberwithin{equation}{section}
\theoremstyle{plain}
\newtheorem{thm}{Theorem}[section]

\newtheorem{cor}[thm]{Corollary}
\newtheorem{lem}[thm]{Lemma}

\theoremstyle{definition}
\newtheorem{defn}[thm]{Definition}
\newtheorem{rem}[thm]{Remark}

\theoremstyle{defi}

\begin{document}

\begin{frontmatter}

\title{Explicit solutions for linear variable--coefficient fractional differential equations with
respect to functions}

\author[kaza,colo,almaty]{Joel E. Restrepo}
\ead{cocojoel89@yahoo.es; joel.restrepo@nu.edu.kz}
\author[mi1,mi2]{Michael Ruzhansky}
\ead{Michael.Ruzhansky@ugent.be}
\author[kaza]{Durvudkhan Suragan$^{*}$}
\ead{durvudkhan.suragan@nu.edu.kz}

\cortext[cor1]{Corresponding author}

\address[kaza]{Department of Mathematics, Nazarbayev University, Nur-Sultan, Kazakhstan}
\address[colo]{Institute of Mathematics, University of Antioquia, Medellin, Colombia}
\address[almaty]{Institute of Mathematics and Mathematical Modeling, Almaty, Kazakhstan}
\address[mi1]{Department of Mathematics, Ghent University, Belgium}
\address[mi2]{School of Mathematical Sciences, Queen Mary University of London, 	United Kingdom}

\begin{abstract}

\noindent Explicit solutions of differential equations of complex fractional orders with respect to functions and with continuous variable coefficients are established. The representations of solutions are given in terms of some convergent infinite series of fractional integro-differential operators, which can be widely and efficiently used for analytic and computational purposes. In the case of constant coefficients, the solution can be expressed in terms of the multivariate Mittag-Leffler functions. In particular, the obtained result extends the Luchko-Gorenflo representation formula \cite[Theorem 4.1]{luchko} to a general class of linear fractional differential equations with variable coefficients, to complex fractional derivatives, and to fractional derivatives with respect to a given function.

\end{abstract}


\end{frontmatter}


\section{Introduction}

In the recent years, the fractional derivatives became an important tool for modelling a variety of physical problems. They take into account the memory effect and other physical characteristics of a modelled process.

One of the many open problems in fractional calculus is to present explicit solutions of fractional differential equations (FDEs) with variable coefficients. A few papers have been published in this direction in the last fifty years, see e.g. \cite{AML,vcserbia1,aa2007vc,RL}. Note that existence and uniqueness results for these types of problems can be found in the literature.  To the best of our knowledge, the first paper \cite{first} on explicit solutions was published in 1968, and then some significant results continue this study, such as those presented before, also  \cite{luchko,vcserbia2,2007vc,cucho,vcapl,analitical}, among others.

In this paper, we use a modification of the method of successive approximations to give explicit representations of solutions for a general class of FDEs of complex fractional orders with respect to functions and with continuous variable coefficients. The obtained solutions are given explicitly in terms of some convergent infinite series of fractional integro-differential operators. Surprisingly, the method of successive approximations has been frequently  used in FDEs with constant coefficients (see e.g.\cite{kilbas2000,kilbas2004,kilbas,aa2007vc}) but not much in FDEs with variable coefficients (see, e.g. \cite{RL,analitical}). In fact, we transform FDEs with variable coefficients to equivalent integral equations, for which the existence and uniqueness are proved by the contractive mapping method. For more details and expository discussions of the topic we refer to \cite[Sections 3 and 4]{kilbas}, and the books \cite{[27],samko}.   

Thus, we combine the latter classical approach with a modified method of successive approximations to establish a unique analytic solution of general FDEs of complex fractional orders with continuous variable coefficients. 

So, we study a new direction of investigation on construction of explicit solutions of fractional differential equations with variable coefficients and some of the most general fractional integro-differential operators, i.e. fractional integro-differential operators with respect to a given function. One of the main advantages of these operators is that they cover many well-known fractional operators such as those of Riemann-Liouville, Hadamard, Erd\'ely--Kober, etc \cite{almeida}. In this manuscript there is also an application of the Banach fixed point theorem in the theory of fractional differential equations, which seems an indispensable tool to fill the gap in the investigation of differential equations with fractional operators.

A particular case of the obtained result gives, for instance, the Luchko-Gorenflo representation formula \cite[Theorem 4.1]{luchko}. Indeed, the results of this paper extend the Luchko-Gorenflo representation formula in several directions:
\begin{itemize}
\item to a general class of linear fractional differential equations with variable coefficients;
\item to fractional derivatives with respect to a given function;
\item and to complex fractional orders.  
\end{itemize}
The latter concepts were introduced by Kober in \cite{kober} and Love in \cite{[11],[12]}. This type of fractional operators played an important role for hypergeometric integral equations whose solutions involved fractional integrals and derivatives of complex orders, see e.g. classical works \cite{[11],[12],ross}, recent papers \cite{r1,r2,r3,r4} and references therein.
These formulas play a crucial role in the analysis of solutions to a variety of fractional problems. For example, in \cite{AKMR, KMR} they have been used for the analysis of equations with singularities, in \cite{RTT} for the analysis of multi-term diffusion-wave equations, in \cite{RTTi} for the inverse problems for subdiffusion equations.

The structure of this paper is given as follows: Section \ref{preli} is devoted to collecting definitions and some basic results on fractional calculus and fractional differential equations.  In Section \ref{main}, we present the main results. First, we concentrate on the canonical sets of solutions of the initial value problem \eqref{eq1} under the homogeneous initial conditions \eqref{eq4}. Then we establish explicit representations of solutions for the general setting, i.e. for the initial value problem \eqref{eq1} under the conditions \eqref{eq2}. We also demonstrate a particular example. Section \ref{remarks} discusses the consequences of the obtained results and compares them with some previously known important results. 

\section{Preliminaries}\label{preli}
In this section we discuss some basic definitions and auxiliary results on fractional calculus. 

\subsection{Fractional calculus} Here we briefly recall the definitions and properties of the fractional integro-differential operators with respect to another function, see e.g. \cite{kilbas} and also \cite[Chapter 4]{samko}. For further discussion we refer to  \cite{oslerapli,almeida,almeidasecond,RS}. We recommend the following papers on applications of fractional calculus with respect to functions \cite{GOPapli,WRZAapli,GGMapli}.

Below, as usual, $C^1[a,b]$ denotes the space of functions whose derivatives
are bounded and continuous in $[a, b].$

\begin{defn} 
Let $\alpha\in\mathbb{C}$, $\rm{Re}(\alpha)>0$, $-\infty\leqslant a<b\leqslant\infty$, let $f$  be an integrable function defined on $[a,b]$, and let $\phi\in C^1[a,b]$ be such that $\phi^{\prime}(t)>0$ for all $t\in[a,b]$. The left-sided Riemann-Liouville fractional integral of a function $f$ with respect to another function $\phi$ is defined as \cite[Formula 2.5.1]{kilbas}
\begin{equation}\label{fraci}
I_{a+}^{\alpha,\phi}f(x)=\frac1{\Gamma(\alpha)}\int_a^x \phi^{\prime}(t)(\phi(x)-\phi(t))^{\alpha-1}f(t)dt.
\end{equation}
\end{defn}
\begin{defn}
Let $\alpha\in\mathbb{C}$, $\rm{Re}(\alpha)>0$, $-\infty\leqslant a<b\leqslant\infty$, let $f$ be an integrable function defined on $[a,b]$, and let $\phi\in C^n[a,b]$ be such that $\phi^{\prime}(t)>0$ for all $t\in[a,b]$. The left-sided Riemann-Liouville fractional derivative of a function $f$ with respect to another function $\phi$ is defined as \cite[Formula 2.5.17]{kilbas}
\begin{equation}\label{fracd}
D_{a+}^{\alpha,\phi}f(x)=\left(\frac{1}{\phi^{\prime}(x)}\frac{d}{dx}\right)^n \big(I_{a+}^{n-\alpha,\phi}f\big)(x),
\end{equation}
where $n=\lfloor\rm{Re}(\alpha)\rfloor +1$ (or $n=-\lfloor-\rm{Re}(\alpha)\rfloor)$ and  $\lfloor \cdot\rfloor$ is the floor function. It must be clear that $n-1<\rm{Re}(\alpha)\leqslant{\it n}$.
\end{defn}
In some particular cases, \eqref{fracd} becomes classical fractional operators. For example, for $\phi(t)=t$, we obtain the  Caputo fractional derivative \cite{caputo}:   
\[^{C}D^{\alpha,x}_{a+} {f(x)}=\frac1{\Gamma(n-\alpha)}\int_a^x (x-t)^{n-\alpha-1}f^{(n)}(t)dt.\]
While for the case $\phi(t)=\ln t$, we get the Caputo-Hadamard fractional derivative (see e.g. \cite{otronew1}):
\[^{C}D^{\alpha,\ln x}_{a+} {f(x)}=\frac1{\Gamma(n-\alpha)}\int_a^x \frac{1}{t}\left(\ln \frac{x}{t}\right)^{n-\alpha-1}\left(t\frac{d}{dt}\right)^{n}f(t)dt.\]
Everywhere in this paper when we refer to the operators $I_{a+}^{\alpha,\phi}$ or $D_{a+}^{\alpha,\phi}$, we assume $\phi\in C^1[a,b]$ such that $\phi^{\prime}(t)>0$ for all $t\in[a,b]$. The definitions above can be considered as some generalizations of the Riemann-Liouville fractional integro-differentiation operators containing a large class of functions. Thus, it allows establishing analogous properties of the Riemann-Liouville operators. Taking into account \cite[Theorem 2.4]{samko} one obtains the following results. 
\begin{thm}\label{dirl}
If $\alpha\in\mathbb{C}$ $(\rm{Re}(\alpha)>0)$ and $f\in L^1(a,b)$, then 
\[D_{a+}^{\alpha,\phi}I_{a+}^{\alpha,\phi}f(x)=f(x)\]
holds almost everywhere on $[a,b]$.
\end{thm}

\begin{thm}\label{idrl}
Let $\alpha\in\mathbb{C}$ $(\rm{Re}(\alpha)>0)$ and $n =\lfloor\alpha\rfloor+ 1$. If $f\in L^1(a,b)$ and $I_{a+}^{n-\alpha,\phi}f(x)\in AC^n[a,b]$, then
\[\big(I_{a+}^{\alpha,\phi}D_{a+}^{\alpha,\phi}f\big)(x)=f(x)-\sum_{j=0}^{n-1}\frac{(D_{a+}^{\alpha-j-1}f)(a)}{\Gamma(\alpha-j)}(\phi(x)-\phi(a))^{\alpha-j-1}\]
holds almost everywhere on $[a,b]$. 
\end{thm}

The following statement can be proved by using Theorem \ref{dirl} and the semi-group property of the operator $I_{a+}^{\alpha,\phi}$ (see \cite[Property 5]{kilbas2000}).
\begin{thm}\label{dirlotro}
Let $\alpha,\beta\in\mathbb{C}$, $\rm{Re}(\alpha)>\rm{Re}(\beta)>0$. For all $f\in L^1(a,b)$, 
\[D_{a+}^{\beta,\phi}I_{a+}^{\alpha,\phi}f(x)=I_{a+}^{\alpha-\beta,\phi}f(x)\]
holds almost everywhere on $[a,b]$.
\end{thm}
Note that in \cite[Formula 1.3]{kilbas2004} the authors used a modified fractional derivative by means of the Riemann-Liouville fractional derivative. Throughout this paper, we use the following modified fractional derivative with respect to another function (cf. \cite[Formula 1.3]{kilbas2004}):

\begin{equation}\label{alternative}
^{C}D_{0+}^{\alpha,\phi}f(x)=\left(D_{0+}^{\alpha,\phi}\left[f(t)-\sum_{j=0}^{n-1}\frac{(D_{a+}^{j,\phi}f)(0)}{j!}\big(\phi(x)-\phi(0)\big)^j\right]\right)(x),\,\alpha\in\mathbb{C},\,\rm{Re}(\alpha)>0,
\end{equation}

where $n=-[-\rm{Re}(\alpha)]$ for $\alpha\notin\mathbb{N}$ and $n=\alpha$ for $\alpha\in\mathbb{N}$

The above type of modified fractional derivatives will allow to gather a wide range of special spaces of functions, which normally require less and weaker conditions for existence of the considered fractional differential operators. Notice also that the large class of fractional derivatives given by the operator \eqref{alternative} will variate from their initial conditions (when $\alpha=1$ in \eqref{alternative}) and even from behavior of the Riemann--Liouville fractional integrals. In fact, in \cite{physical}, the imposed  initial conditions over some considered fractional differential equations have several different interpretations and meanings in nature of physical phenomena or/and operators.

If $\alpha>0$, $n-1<\alpha<n$ and $f\in C^n[a,b]$, then \eqref{alternative} gives the Caputo fractional derivative \cite[Theorem 3]{almeida}, i.e. 
\begin{equation}\label{Capder}
^{C}D^{\alpha,\phi}_{a+} {f(x)}=I_{a+}^{n-\alpha,\phi}\left(\frac1{\phi^{\prime}(x)}\frac{d}{dx}\right)^n f(x),
\end{equation}
where $n=\lfloor \alpha\rfloor+1$ for $\alpha\notin\mathbb{N}$ and $n=\alpha$ for $\alpha\in\mathbb{N}$. Here $f^{(n)}\in L^1[a,b]$ is sufficient for the existence of the Caputo fractional derivative \eqref{Capder}, and $f\in C^n[a,b]$ ensures the continuity for the Caputo derivative. Moreover, if $\alpha=n$, then one has $^{C}D_{0+}^{\alpha,\phi}f(t)=D^{n}f(t)=f^{(n)}(t).$

The following result follows the same steps of the proof of the necessity in \cite[Theorem 1]{kilbas2004}. Therefore, we omit its proof. 

\begin{thm}\label{intder}
If $\alpha\in\mathbb{C}$, $\rm{Re}(\alpha)>0$, $f\in C^{n-1}[a,b]$, then
$$ I_{a+}^{\alpha,\phi}\, {^C}D_{a+}^{\alpha,\phi}f(x)=f(x)-\sum_{j=0}^{n-1}\frac{(D_{a+}^{j,\phi}f)(a)}{j!}\big(\phi(x)-\phi(a)\big)^j. $$
\end{thm}

We conclude this subsection by the following useful lemma:
\begin{lem}\label{importantpro}
If $\alpha\in\mathbb{C}$, $\rm{Re}(\alpha)>0$ and $f$ is continuous in $[0,b]$ then the following statements hold:
\begin{enumerate}
\item $I_{0+}^{\alpha,\phi}f(x)$ is a continuous function on $[0,b]$.
\item $\displaystyle{\lim_{x\to0+}I_{0+}^{\alpha,\phi}f(x)=0.}$
\item For any $\rm{Re}(\alpha)>\rm{Re}(\beta)>0$, we have  
\[^{C}D_{0+}^{\beta,\phi}I_{0+}^{\alpha,\phi}f(x)=I_{0+}^{\alpha-\beta,\phi}f(x).\]
For $\alpha=\beta$, we have $^{C}D_{0+}^{\alpha,\phi}I_{0+}^{\alpha,\phi}f(x)=f(x).$
\end{enumerate}
\end{lem}
\begin{proof}
The first assertion can be proved directly by using the boundedness of the function $f$ and the integrability of $\phi^\prime(s)(\phi(t)-\phi(s))^{\alpha-1}$ over $0\leqslant s\leqslant t\leqslant b$.  The mean value theorem for integrals implies the second statement. The proof of the last assertion follows from Definition \ref{alternative}, Theorems \ref{dirl}, \ref{dirlotro} and the second statement.
\end{proof}

\subsection{Fractional differential equations} We consider the following fractional differential equation with continuous variable coefficients: 
\begin{equation}\label{eq1}
^{C}D_{0+}^{\beta_0,\phi}x(t)+\sum_{i=1}^{m}d_i(t) ^{C}D_{0+}^{\beta_i,\phi}x(t)=h(t),\quad t\in[0,T],\; m\in\mathbb{N},
\end{equation}
under the initial conditions 
\begin{equation}\label{eq2}
\left(\frac1{\phi^{\prime}(t)}\frac{d}{dt}\right)^k x(t)|_{_{t=+0}}=c_k\in\mathbb{R},\quad k=0,1,\ldots,n_0-1, 
\end{equation}
where $\beta_i\in\mathbb{C}$, $\rm{Re}(\beta_{\it i})>0$, $i=0,1,\ldots,m-1$, $\rm{Re}(\beta_0)>\rm{Re}(\beta_1)>\ldots>\rm{Re}(\beta_{\it m})\geqslant0$ (If $\rm{Re}(\beta_{\it m})=0$, then we assume $\rm{Im}(\beta_{\it m})=0$ as well) and $n_i$ are non-negative integers satisfying $n_i-1<\rm{Re}(\beta_{\it i})\leqslant{\it n_i},$ $n_i=\lfloor \rm{Re}\beta_{\it i}\rfloor+1$ (or $n_i=-\lfloor-\rm{Re}(\beta_{\it i})\rfloor$), $i=0,1,\ldots,m$. We also consider the homogeneous case  
\begin{equation}\label{eq3}
^{C}D_{0+}^{\beta_0,\phi}x(t)+\sum_{i=1}^{m}d_i(t) ^{C}D_{0+}^{\beta_i,\phi}x(t)=0,\quad t\in[0,T],\; m\in\mathbb{N},
\end{equation}
and
\begin{equation}\label{eq4}
\left(\frac1{\phi^{\prime}(t)}\frac{d}{dt}\right)^k x(t)|_{_{t=+0}}=0,\quad k=0,1,\ldots,n_0-1.
\end{equation}
Let us denote the set $\mathbb{K}_j$ by 
\[\mathbb{K}_j:=\{i:0\leqslant\rm{Re}(\beta_{\it i})\leqslant{\it j}, {\it i}=1,\dots,{\it m}\},\quad {\it j}=0,1,\dots,{\it n}_0-1,\]
and $\varkappa_j=\displaystyle{\min\{\mathbb{K}_j\}}$ if $\mathbb{K}_j\neq\emptyset$. Notice that $s\in\mathbb{K}_j$ implies $\rm{Re}(\beta_{\it s})\leqslant\it{j}$ and $\mathbb{K}_{j_1}\subset\mathbb{K}_{j_2}$ for $j_1<j_2$.

Under this notation we notice  the possible cases: 
\begin{enumerate}
\item $\beta_m=0$ or $\beta_{m}=0$. Here it follows that $\mathbb{K}_j\neq\emptyset$ for any $j=0,1,\ldots,n_0-1$.
\item $n_0\geqslant2$ and there exists $j_0\in\{0,1,\ldots,n_0-2\}$ which satisfy $j_0<\rm{Re}(\beta_{\it r})\leqslant {\it j}_0+1$ for some $r\in\{1,\ldots,m\}$. It is equivalent to $\mathbb{K}_{j_0}=\emptyset$ and $\mathbb{K}_{j_0+1}\neq\emptyset$. We then get that $\mathbb{K}_j=\emptyset$ for any $j=0,1,\ldots,j_0$ and $\mathbb{K}_j\neq\emptyset$ for any $j=j_0+1,\ldots,n_0-1$. 
\item $\rm{Re}(\beta_{\it i})>{\it n}_0-1$ for any $i=1,\ldots,m,$ i.e. 
$\mathbb{K}_{n_0-1}=\emptyset$ . Hence $\mathbb{K}_j=\emptyset$ for any $j=0,1,\ldots,n_0-1.$
\end{enumerate}
    
\begin{defn}
A set of functions $x_j(t)$ $(j=0,1,\ldots,n_0)$ is called a canonical set of solutions of equation \eqref{eq3} if it satisfies
\[^{C}D_{0+}^{\beta_0,\phi}x_j(t)=-\sum_{i=1}^{m}d_i(t) ^{C}D_{0+}^{\beta_i,\phi}x_j(t),\quad 0<t<T,\]
and
\[\left(\frac{1}{\phi^{\prime}(t)}\frac{d}{dt}\right)^k x_j(t)|_{_{t=+0}} = \left\{
\begin{array}{cl}
1,& j=k,\\
0,& j\neq k,\,\,j=0,1,\ldots,n_0-1.
\end{array}\right.\]
\end{defn}

\section{Main results}\label{main}
We begin this section by introducing a special function space to be used in our analysis:
\[C^{n_0-1,\beta_0}[0,T]:=\{x(t)\in C^{n_0-1}[0,T],\, ^{C}D_{0+}^{\beta_0,\phi}x(t)\in C[0,T]\}\]
endowed with the norm
\[\|x\|_{C^{n_0-1,\beta_0}[0,T]}=\sum_{k=0}^{n_0-1}\left\|\left(\frac1{\phi^{\prime}(t)}\frac{d}{dt}\right)^k x\right\|_{C[0,T]}+\big\|\,^{C}D_{0+}^{\beta_0,\phi}x\big\|_{C[0,T]}.\] 
\subsection{Canonical sets of solutions}
Here we prove the existence and uniqueness of solutions for the
initial value problem \eqref{eq1} with conditions \eqref{eq4}, where the solution is given by the limit of a convergent sequence. 
\begin{lem}\label{lem3.1}
Let $h,d_i\in C[0,T]$, $i=1,\ldots,m$. Then the initial value problem \eqref{eq1} and \eqref{eq4} has a unique solution $x\in C^{n_0-1,\beta_0}[0,T]$, and it is given by the limit  $\displaystyle{x(t)=\lim_{n\to+\infty}x_n(t)}$ of the  sequence 
\begin{equation}\label{eq5eq6}
\left\{\begin{split}
x_0(t)&=I_{0+}^{\beta_0,\phi}h(t), \\
x_n(t)&=x_0(t)-I_{0+}^{\alpha_0}\sum_{i=1}^{m}d_i(t)\,^{C}D_{0+}^{\beta_i,\phi}x_{n-1}(t),\quad n=1,2,\ldots, 
\end{split}
\right.
\end{equation}  
whenever $\displaystyle{\sum_{i=1}^m \|d_i\|_{\max}I_{0+}^{\beta_0-\beta_i,\phi}e^{\nu t}}\leqslant Ce^{\nu t}$ for any $t\in[0,T]$, some fixed $\nu\in\mathbb{R}_+$ and a constant  $0<C<1$, which does not depend on $t$.
\end{lem}
\begin{proof}
First, we show that the initial value problem \eqref{eq1} and \eqref{eq4} is equivalent to an integral equation. Suppose that $x(t)\in C^{n_0-1,\beta_0}[0,T]$ satisfies \eqref{eq1} and \eqref{eq4}. Setting $w(t)=\,^{C}D_{0+}^{\beta_0,\phi}x(t)$, we have $w(t)\in C[0,T]$ since $x(t)\in C^{n_0-1,\beta_0}[0,T]$. Moreover, we obtain
\begin{equation}\label{va1}
I_{0+}^{\beta_0,\phi}w(t)=I_{0+}^{\beta_0,\phi}\,^{C}D_{0+}^{\beta_0,\phi}x(t)=x(t).
\end{equation}
Here we have used Theorem \ref{intder} and the initial conditions \eqref{eq4}. Using the fact $w(t)\in C[0,T]$, $\rm{Re}(\beta_0)>\rm{Re}(\beta_{\it i})>0$ and Lemma \ref{importantpro} we get 
\[^{C}D_{0+}^{\beta_i,\phi}x(t)=\,^{C}D_{0+}^{\beta_i,\phi}I_{0+}^{\beta_0,\phi}w(t)=I_{0+}^{\beta_0-\beta_i,\phi}w(t),\]
for any $i=1,\ldots,m$. If $\rm{Re}(\beta_{\it m})=0$, then $\beta_m=0$,  and the equality above follows immediately. Thus,  equation \eqref{eq1} becomes
\begin{equation}\label{integraleq}
w(t)+\sum_{i=1}^{m}d_i(t)I_{0+}^{\beta_0-\beta_i,\phi}w(t)=h(t).
\end{equation}
So, if $x(t)\in C^{n_0-1,\beta_0}[0,T]$ is a solution of problem \eqref{eq1} and \eqref{eq4}, then $w(t)=\,^{C}D_{0+}^{\beta_0,\phi}x(t)\in C[0,T]$ is the solution of integral equation \eqref{integraleq}.

Let us now prove the converse statement. Assume  that $w(t)\in C[0,T]$ is the solution of \eqref{integraleq}. Applying the operator $I_{0+}^{\alpha_0,\phi}$ to expression \eqref{integraleq} we obtain 
\[I_{0+}^{\beta_0,\phi}w(t)+I_{0+}^{\beta_0,\phi}\sum_{i=1}^{m}d_i(t)I_{0+}^{\beta_0-\beta_i,\phi}w(t)=I_{0+}^{\beta_0,\phi}h(t).\]
Let $x(t)=I_{0+}^{\beta_0,\phi}w(t)$. By Lemma \ref{importantpro} we have $\,^{C}D_{0+}^{\beta_i,\phi}x(t)=I_{0+}^{\beta_0-\beta_i,\phi}w(t)$ and $I_{0+}^{\beta_0-\beta_i,\phi}w(t)\in C[0,T]$. It implies
\[x(t)+I_{0+}^{\beta_0,\phi}\sum_{i=1}^{m}d_i(t) ^{C}D_{0+}^{\beta_i,\phi}x(t)=I_{0+}^{\beta_0,\phi}h(t).\]
Thus, 
\[^{C}D_{0+}^{\beta_0,\phi}x(t)+\,^{C}D_{0+}^{\beta_0,\phi}I_{0+}^{\beta_0,\phi}\sum_{i=1}^{m}d_i(t) ^{C}D_{0+}^{\beta_i,\phi}x(t)=\,^{C}D_{0+}^{\beta_0,\phi}I_{0+}^{\beta_0,\phi}h(t).\] 
From Lemma \ref{importantpro} it follows that 
\[^{C}D_{0+}^{\beta_0,\phi}x(t)+\sum_{i=1}^{m}d_i(t) ^{C}D_{0+}^{\beta_i,\phi}x(t)=h(t).\] 
Also, by Lemma \ref{importantpro} and $\rm{Re}(\beta_0)>{\it n}_0-1$ we have
\[\left(\frac1{\phi^{\prime}(t)}\frac{d}{dt}\right)^k x(t)|_{_{t=+0}}=\left(\frac1{\phi^{\prime}(t)}\frac{d}{dt}\right)^k I_{0+}^{\beta_0,\phi}w(t)|_{_{t=+0}}=I_{0+}^{\beta_0-k,\phi}w(t)|_{_{t=+0}}=0,\]
for any $k=0,1,\ldots,n_0-1$, and $^{C}D_{0+}^{\beta_0,\phi}x(t)=\,^{C}D_{0+}^{\beta_0,\phi}I_{0+}^{\beta_0,\phi}w(t)=w(t)\in C[0,T]$. Then the solution $w(t)\in C[0,T]$ of \eqref{integraleq} implies that $x(t)=I_{0+}^{\beta_0,\phi}w(t)\in C^{n_0-1,\beta_0}[0,T]$ is the solution of problem \eqref{eq1} satisfying conditions \eqref{eq4}. 
Hence initial value problem \eqref{eq1} and \eqref{eq4} is equivalent to integral equation \eqref{integraleq}.

Now we prove the existence and uniqueness of a solution of equation \eqref{integraleq}. From \eqref{integraleq} we have
\begin{equation}\label{integraleq8}
w(t)=h(t)-\sum_{i=1}^{m}d_i(t)I_{0+}^{\beta_0-\alpha_i,\phi}w(t).
\end{equation}  
Let us define the operator $T$ by 
\[Tw(t):=h(t)-\sum_{i=1}^{m}d_i(t)I_{0+}^{\beta_0-\beta_i,\phi}w(t).\]
By \eqref{integraleq8} we know that $Tw(t)=w(t)$, so $T:C[0,T]\to C[0,T]$. From now on we use an equivalent norm to the maximum norm on $C[0,T]$, i.e. $\|z\|_{\nu}:=\displaystyle{\max_{t\in[0,T]}\{e^{-\nu t}|z(t)|\}}$ for the fixed $\nu\in\mathbb{R}_+$. Since $\displaystyle{\sum_{i=1}^m \|d_i\|_{\max}I_{0+}^{\beta_0-\beta_i,\phi}e^{\nu t}}\leqslant Ce^{\nu t}$ for some $0<C<1$, it follows for any $t\in [0,T]$ that 
\begin{align*}
|Tw_1(t)-Tw_2(t)|&\leqslant\sum_{i=1}^{m}\|d_i\|_{\max}I_{0+}^{\beta_0-\beta_i,\phi}\big|w_1(t)-w_2(t)\big| \\
&\leqslant\|w_1-w_2\|_{\nu}\sum_{i=1}^{m}\|d_i\|_{\max}I_{0+}^{\beta_0-\beta_i,\phi}e^{\nu t} \\
&\leqslant C\|w_1-w_2\|_{\nu}e^{\nu t}.
\end{align*}
This yields 
\[e^{-\nu t}|Tw_1(t)-Tw_2(t)|\leqslant C\|w_1-w_2\|_{\nu}\Rightarrow \|Tw_1-Tw_2\|_{\nu}\leqslant C\|w_1-w_2\|_{\nu}. \]
It means that the operator $T$ is contractive with respect to the norm $\|\cdot\|_{\nu}$, hence it is contractive with respect to the maximum norm $\|\cdot\|_{\max}$. Thus, by the Banach fixed point theorem, integral equation \eqref{integraleq} has a unique solution with respect to $\|\cdot\|_{\max}$ and the sequence $\{w_n(t)\}_{n\geqslant0}$ given by $w_n(t)=h(t)-\displaystyle{\sum_{i=1}^{m}d_i(t)I_{0+}^{\beta_0-\beta_i,\phi}w_{n-1}(t)}$ converges with respect to $\|\cdot\|_{\max}$ in $C[0,T]$.    

It remains to prove that the sequence \eqref{eq5eq6} converges in $C^{n_0-1,\beta_0}[0,T]$ with respect to the norm $\|\cdot\|_{C^{n_0-1,\beta_0}[0,T]}$. It is clear that 
\begin{equation*}
\left\{\begin{split}
w_0(t)&=h(t), \\
w_n(t)&=h(t)-\sum_{i=1}^{m}d_i(t)I_{0+}^{\beta_0-\beta_i,\phi}w_{n-1}(t), 
\end{split}
\right.
\end{equation*}
converges with respect to $\|\cdot\|_{\max}$. Thus, we have
\begin{equation*}
\left\{\begin{split}
I_{0+}^{\beta_0,\phi}w_0(t)&=I_{0+}^{\beta_0,\phi}h(t), \\
I_{0+}^{\beta_0,\phi}w_n(t)&=I_{0+}^{\beta_0,\phi}h(t)-I_{0+}^{\beta_0,\phi}\sum_{i=1}^{m}d_i(t)I_{0+}^{\beta_0-\beta_i,\phi}w_{n-1}(t). 
\end{split}
\right.
\end{equation*}
We set $x_{n}(t)=I_{0+}^{\beta_0,\phi}w_n(t)$. Then $^{C}D_{0+}^{\beta_i,\phi}x_{n-1}(t)=I_{0+}^{\beta_0-\beta_i,\phi}w_{n-1}(t)$ and
\begin{equation*}
\left\{\begin{split}
x_0(t) &=I_{0+}^{\beta_0,\phi}h(t), \\
x_n(t)&=x_0(t)-I_{0+}^{\beta_0,\phi}\sum_{i=1}^{m}d_i(t)\,^{C}D_{0+}^{\beta_i,\phi}x_{n-1}(t). 
\end{split}
\right.
\end{equation*}
The sequence above is the same as in \eqref{eq5eq6}. It is clear that $x_n(t)\in C^{n_0-1,\beta_0}[0,T]$. 

Now we prove the convergence of $\{x_n(t)\}_{n\geqslant0}$ in $C^{n_0-1,\beta_0}[0,T]$. Indeed, since $x_n(t)=I_{0+}^{\beta_0,\phi}w_n(t)$ and  $^{C}D_{0+}^{\beta_0,\phi}x_n(t)=w_n(t)$ we obtain  
\[\left(\frac1{\phi^{\prime}(t)}\frac{d}{dt}\right)^k x_n(t)=I_{0+}^{\beta_0-k,\phi}w_n(t),\quad k=0,1,\ldots,n_0-1.\]
Hence $\|\,^{C}D_{0+}^{\beta_0,\phi}x_n\|_{\max}=\|w_n\|_{\max},$ and 
\[\left\|\left(\frac1{\phi^{\prime}(t)}\frac{d}{dt}\right)^k x_n\right\|_{\max}\leqslant \frac{T^{\rm{Re}(\beta_0)-{\it k}}}{\Gamma(\rm{Re}(\beta_0)-{\it k}+1)}\|w_n\|_{\max},\quad k=0,1,\ldots,n_0-1.\]
Thus, we arrive at
\[\sum_{k=1}^{n_0-1}\left\|\left(\frac1{\phi^{\prime}(t)}\frac{d}{dt}\right)^k x_n\right\|_{\max}+\|\,^{C}D_{0+}^{\alpha_0,\phi}x_n\|_{\max}\leqslant \left(\sum_{k=0}^{n_0-1}\frac{T^{\rm{Re}(\beta_0)-{\it k}}}{\Gamma(\rm{Re}(\beta_0)-{\it k}+1)}+1\right)\|w_n\|_{\max}.\]
Since the sequence $\{w_n(t)\}_{n\geqslant0}$ converges with respect to $\|\cdot\|_{\max}$, then the sequence $\{x_n(t)\}_{n\geqslant0}$ converges in $C^{n_0-1,\beta_0}[0,T]$ with respect to $\|\cdot\|_{C^{n_0-1,\beta_0}[0,T]}.$ It completes the proof.
\end{proof}
\begin{rem}\label{newcond}
Let us give an example about the assumption
\[\displaystyle{\sum_{i=1}^m \|d_i\|_{\max}I_{0+}^{\beta_0-\beta_i,\phi}e^{\nu t}}\leqslant Ce^{\nu t}\] 
being fulfilled under consideration of some well-known function $\phi$. For this purpose, let us consider the case $\phi(x)=x$. In this case, we have $I_{0+}^{\beta_0-\beta_i}e^{k^{*}t}\leqslant \frac{e^{k^{*}t}}{(k^*)^{\beta_0-\beta_i}}$ for any $i=1,\ldots,m$ and any $k^*\in\mathbb{R}_+$. Indeed
\begin{align*}
I_{0+}^{\beta_0-\beta_i}e^{k^{*}t}&=\frac{1}{\Gamma(\beta_0-\beta_i)}\int_0^t (t-x)^{\beta_0-\beta_i-1}e^{k^{*}x}dx=\frac{e^{k^{*}t}}{\Gamma(\beta_0-\beta_i)}\int_0^t u^{\beta_0-\beta_i-1}e^{-k^{*}u}du \\
&\leqslant \frac{e^{k^{*}t}}{\Gamma(\beta_0-\beta_i)\,(k^{*})^{\beta_0-\beta_1}}\int_0^{+\infty}r^{\beta_0-\beta_i-1}e^{-r}dr=\frac{e^{k^{*}t}}{(k^{*})^{\beta_0-\beta_1}}.
\end{align*}
So, there exists $\nu\in\mathbb{R}_+$ such that for any $k^{*}>\nu$ we have 
\[\sum_{i=1}^{m}\|d_i\|_{\max}\frac{1}{(k^*)^{\beta_0-\beta_i}}\leqslant\sum_{i=1}^{m}\|d_i\|_{\max}\frac{1}{\nu^{\beta_0-\beta_i}}<1\]
and
\[\sum_{i=1}^m \|d_i\|_{\max}I_{0+}^{\beta_0-\beta_i,\phi}e^{\nu t}\leqslant Ce^{\nu t}.\]
\end{rem}
Now we focus on finding the canonical set of solutions of  homogeneous equation \eqref{eq3} under different cases of $\mathbb{K}_j$ $(j=0,\ldots,n_0-1)$ and the relation between $n_0$ and $n_1$. 
\begin{lem}\label{lem3.3}
Let $n_0>n_1$, $d_i\in C[0,T]$ for any $i=1,\ldots,m$ and $\beta_{m}=0$. Then there exists a unique canonical set $x_j\in C^{n_0-1,\beta_0}[0,T]$, $j=0,1,\ldots,n_0-1$, of solutions of equation \eqref{eq3} in the form
\begin{equation}\label{form16}
x_j(t)=\Psi_j(t)+\sum_{k=0}^{+\infty} (-1)^{k+1} I_{0+}^{\beta_0,\phi}\left(\sum_{i=1}^{m}d_i(t) I_{0+}^{\beta_0-\beta_i,\phi}\right)^k \sum_{i=\varkappa_j}^{m}d_i(t)D_{0+}^{\beta_i,\phi}\Psi_j(t),
\end{equation} 
for  $j=0,1,\ldots,n_1-1$, and 
\begin{equation}\label{form17}
x_j(t)=\Psi_j(t)+\sum_{k=0}^{+\infty} (-1)^{k+1} I_{0+}^{\beta_0,\phi}\left(\sum_{i=1}^{m}d_i(t) I_{0+}^{\beta_0-\beta_i,\phi}\right)^k \sum_{i=1}^{m}d_i(t)D_{0+}^{\beta_i,\phi}\Psi_j(t),
\end{equation} 
for $j=n_1,n_1+1,\ldots,n_0-1$, where $\Psi_j(t)=\frac{(\phi(t)-\phi(0))^{j}}{\Gamma(j+1)}$.
\end{lem}
\begin{proof}
We show that the canonical set of solutions of equation \eqref{eq3} is given by the limit of the sequence:
\begin{equation}\label{eq10eq11}
\left\{\begin{split}
x^0_j (t)&=\Psi_j(t), \\
x^n_j (t)&=x^0_j(t)-I_{0+}^{\beta_0,\phi}\sum_{i=1}^{m}d_i(t)\,^{C}D_{0+}^{\beta_i,\phi}x^{n-1}_j(t),\quad n=1,2,\ldots, 
\end{split}
\right.
\end{equation} 
for any $j=0,1,\ldots,n_0-1.$ Notice that the sequence above appears naturally from equation \eqref{eq3} taking into account the inverse operations between $^{C}D_{0+}^{\beta_0,\phi}$, $I_{0+}^{\beta_0,\phi}$, Theorem \ref{intder} and Lemma \ref{importantpro}. Let us fix $j\in\{0,1,\ldots,n_0-1\}$ and find the $jth$ term $x_j(t)$ of the canonical set. Note that for $k=0,1,\ldots,n_i-1$ and $m\geqslant i\geqslant1$, we have 
\begin{equation*}
\left(\frac{1}{\phi^{\prime}(t)}\frac{d}{dt}\right)^k \Psi_j(t)|_{_{_{t=0+}}}=
\left\{\begin{split}
1,&\quad k=j, \\
0,&\quad k\neq j. 
\end{split}
\right.
\end{equation*}
Thus, for $j=0,1,\ldots,n_1-1$ and $n_0>n_1\geqslant n_i$ we get
\begin{equation}\label{formula18}
^{C}D_{0+}^{\beta_i,\phi}\Psi_j(t)=
\left\{\begin{split}
D_{0+}^{\beta_i,\phi}\Psi_j(t),&\quad \varkappa_j\leqslant i\leqslant m\, (j\geqslant n_i), \\
0,&\quad 1\leqslant i< \varkappa_j\, (n_i>j). 
\end{split}
\right.
\end{equation}
For $j=n_1,\ldots,n_0-1$ we have $k<j$ and
\[^{C}D_{0+}^{\beta_i,\phi}\Psi_j(t)=D_{0+}^{\beta_i,\phi}\Psi_j(t),\quad i=1,\ldots,m.\]
Now we obtain the first approximation of $x_j(t)$ by
\begin{align*}
x^1_j (t)&=\Psi_j(t)-I_{0+}^{\beta_0,\phi}\sum_{i=\varkappa_j}^{m}d_i(t)\,^{C}D_{0+}^{\beta_i,\phi}\Psi_j(t)=\Psi_j(t)-I_{0+}^{\beta_0,\phi}\sum_{i=\varkappa_j}^{m}d_i(t)D_{0+}^{\beta_i,\phi}\Psi_j(t),
\end{align*}
$j=0,1,\ldots,n_1-1$, and
\begin{align*}
x^1_j (t)&=\Psi_j(t)-I_{0+}^{\beta_0,\phi}\sum_{i=1}^{m}d_i(t)\,^{C}D_{0+}^{\beta_i,\phi}\Psi_j(t)=\Psi_j(t)-I_{0+}^{\beta_0,\phi}\sum_{i=1}^{m}d_i(t)D_{0+}^{\beta_i,\phi}\Psi_j(t),
\end{align*}
$j=n_1,n_1+1,\ldots,n_0-1.$ Therefore, it follows that $x_j^1(t)\in C^{n_0-1,\beta_0}[0,T]$ for any $j=0,1,\ldots,n_0-1$, see Lemma \ref{importantpro}. For $j=0,1,\ldots,n_1-1$, we have the second approximation as follows
\begin{align*}
x^2_j (t)&=\Psi_j(t)-I_{0+}^{\beta_0,\phi}\sum_{i=1}^{m}d_i(t)\,^{C}D_{0+}^{\beta_i,\phi}x^{1}_j(t) \\
&=\Psi_j(t)-I_{0+}^{\beta_0,\phi}\sum_{i=1}^{m}d_i(t)\,^{C}D_{0+}^{\beta_i,\phi}\Psi_j(t)+I_{0+}^{\beta_0,\phi}\sum_{i=1}^{m}d_i(t)\,^{C}D_{0+}^{\beta_i,\phi}I_{0+}^{\beta_0,\phi}\sum_{i=\varkappa_j}^{m}d_i(t)D_{0+}^{\beta_i,\phi}\Psi_j(t).
\end{align*}
Due to Lemma \ref{importantpro} we have 
 \[^{C}D_{0+}^{\beta_i,\phi}I_{0+}^{\beta_0,\phi}\sum_{i=\varkappa_j}^{m}d_i(t)D_{0+}^{\beta_i,\phi}\Psi_j(t)=I_{0+}^{\beta_0-\beta_i,\phi}\sum_{i=\varkappa_j}^{m}d_i(t)D_{0+}^{\beta_i,\phi}\Psi_j(t).\] 
Combining the latter equality, formula \eqref{formula18} and the above representation of $x_j^2(t)$ we get
\begin{align*}
x^2_j (t)&=\Psi_j(t)-I_{0+}^{\beta_0,\phi}\sum_{i=\varkappa_j}^{m}d_i(t)D_{0+}^{\beta_i,\phi}\Psi_j(t)+I_{0+}^{\beta_0,\phi}\sum_{i=1}^{m}d_i(t)I_{0+}^{\beta_0-\beta_i,\phi}\sum_{i=\varkappa_j}^{m}d_i(t)D_{0+}^{\beta_i,\phi}\Psi_j(t) \\
&=\Psi_j(t)+\sum_{k=0}^{1}(-1)^{k+1}I_{0+}^{\beta_0,\phi}\left(\sum_{i=1}^{m}d_i(t)I_{0+}^{\beta_0-\beta_i,\phi}\right)^{k}\sum_{i=\varkappa_j}^{m}d_i(t)D_{0+}^{\beta_i,\phi}\Psi_j(t).
\end{align*}
Similarly, for $j=n_1,n_1+1,\ldots,n_0-1$ we obtain the second approximation by
\begin{align*}
x^2_j (t)&=\Psi_j(t)-I_{0+}^{\beta_0,\phi}\sum_{i=1}^{m}d_i(t)\,^{C}D_{0+}^{\beta_i,\phi}x^{1}_j(t) \\
&=\Psi_j(t)+\sum_{k=0}^{1}(-1)^{k+1}I_{0+}^{\beta_0,\phi}\left(\sum_{i=1}^{m}d_i(t)I_{0+}^{\beta_0-\beta_i,\phi}\right)^{k}\sum_{i=1}^{m}d_i(t)D_{0+}^{\beta_i,\phi}\Psi_j(t).
\end{align*}
Hence the second approximation of $x_j(t)$ is given by 
\begin{align*}
x^2_j (t)=\Psi_j(t)+\sum_{k=0}^{1}(-1)^{k+1}I_{0+}^{\beta_0,\phi}\left(\sum_{i=1}^{m}d_i(t)I_{0+}^{\beta_0-\beta_i,\phi}\right)^{k}\sum_{i=\varkappa_j}^{m}d_i(t)D_{0+}^{\beta_i,\phi}\Psi_j(t),
\end{align*}
for $j=0,1,\ldots,n_1-1,$ and
\begin{align*}
x^2_j (t)=\Psi_j(t)+\sum_{k=0}^{1}(-1)^{k+1}I_{0+}^{\beta_0,\phi}\left(\sum_{i=1}^{m}d_i(t)I_{0+}^{\beta_0-\beta_i,\phi}\right)^{k}\sum_{i=1}^{m}d_i(t)D_{0+}^{\beta_i,\phi}\Psi_j(t),
\end{align*}
for $j=n_1,n_1+1,\ldots,n_0-1$. Also we have $x_j^2(t)\in C^{n_0-1,\beta_0}[0,T]$ for any $j=0,1,\ldots,n_0-1$. By using the induction process we can arrive at 
 \begin{align*}
x^n_j (t)=\Psi_j(t)+\sum_{k=0}^{n-1}(-1)^{k+1}I_{0+}^{\beta_0,\phi}\left(\sum_{i=1}^{m}d_i(t)I_{0+}^{\beta_0-\beta_i,\phi}\right)^{k}\sum_{i=\varkappa_j}^{m}d_i(t)D_{0+}^{\beta_i,\phi}\Psi_j(t),
\end{align*}
for $j=0,1,\ldots,n_1-1,$ and
\begin{align*}
x^n_j (t)=\Psi_j(t)+\sum_{k=0}^{n-1}(-1)^{k+1}I_{0+}^{\beta_0,\phi}\left(\sum_{i=1}^{m}d_i(t)I_{0+}^{\beta_0-\beta_i,\phi}\right)^{k}\sum_{i=1}^{m}d_i(t)D_{0+}^{\beta_i,\phi}\Psi_j(t),
\end{align*}
for $j=n_1,n_1+1,\ldots,n_0-1$, and $x_j^n(t)\in C^{n_0-1,\beta_0}[0,T]$ for $j=0,1,\ldots,n_0-1$. By the same argument given at the end of the proof of Lemma \ref{lem3.1} we have  $x_j(t)=\displaystyle{\lim_{n\to+\infty}x_j^n(t)}\in C^{n_0-1,\beta_0}[0,T]$. Thus, taking into account Lemma \ref{importantpro} we verify that the canonical set $x_j\in C^{n_0-1,\beta_0}[0,T]$, $j=0,1,\ldots,n_0-1$, of solutions of equation \eqref{eq3} is given by \eqref{form16} and \eqref{form17}.
\end{proof}

\begin{lem}\label{lem3.2}
Let $n_0=n_1$, $d_i\in C[0,T]$ for any $i=1,\ldots,m$ and $\beta_{m}=0$. Then there exists a unique canonical set $x_j\in C^{n_0-1,\beta}[0,T]$, $j=0,1,\ldots,n_0-1$, of solutions of equation \eqref{eq3} in the form
\begin{equation}\label{form10}
x_j(t)=\Psi_j(t)+\sum_{k=0}^{+\infty} (-1)^{k+1} I_{0+}^{\beta_0,\phi}\left(\sum_{i=1}^{m}d_i(t) I_{0+}^{\beta_0-\beta_i,\phi}\right)^k \sum_{i=\varkappa_j}^{m}d_i(t)D_{0+}^{\beta_i,\phi}\Psi_j(t),
\end{equation} 
for $j=0,1,\ldots,n_0-1$, where $\Psi_j(t)=\frac{(\phi(t)-\phi(0))^{j}}{\Gamma(j+1)}$.
\end{lem}
\begin{proof}
Let us show that the canonical set of solutions of equation \eqref{eq3} is given by the limit in $C^{n_0-1,\beta_0}[0,T]$ of  approximation sequence \eqref{eq10eq11}. We start with the first element $x_0(t)$. From \eqref{eq10eq11} we have $x_0^0(t)=\Psi_0(t)$ and 
\begin{equation}\label{form13} 
x^1_0 (t)=x^0_0(t)-I_{0+}^{\beta_0,\phi}\sum_{i=1}^{m}d_i(t)\,^{C}D_{0+}^{\beta_i,\phi}x^{0}_0(t)=\Psi_0(t)-I_{0+}^{\beta_0,\phi}\sum_{i=1}^{m}d_i(t)\,^{C}D_{0+}^{\beta_i,\phi}\Psi_0(t).
\end{equation}
Since $\beta_{m}=0$ we have $\varkappa_0=\text{min}\{\mathbb{K}_{0}\}=m$. For $i=m$, $\beta_m=n_m=0$ and $^{C}D_{0+}^{\beta_i,\phi}\Psi_0(t)=D_{0+}^{\beta_m,\phi}\Psi_0(t)=\Psi_0(t)$. While, for $i=1,\ldots,m-1$ we have $\rm{Re}(\beta_{\it i})>0$, $n_i-1<\rm{Re}(\beta_{\it i})\leqslant{\it n}_i$, $n_i\geqslant 1$ and $^{C}D_{0+}^{\beta_i,\phi}\Psi_0(t)=0$. Substituting these in \eqref{form13} we get
\[x^1_0(t)=\Psi_0(t)-I_{0+}^{\beta_0,\phi}d_m(t)\Psi_0(t).\]
Due to $\varkappa_0=m$, $\beta_m=0$ and $D_{0+}^{\beta_m,\phi}\Psi_0(t)$ we also obtain
\[x^1_0(t)=\Psi_0(t)-I_{0+}^{\beta_0,\phi}\sum_{i=\varkappa_0}^m d_i(t)D_{0+}^{\beta_i,\phi}\Psi_0(t).\]  
We thus have the first term of formula \eqref{form10} for $j=0$, where $x_0^1(t)\in C^{n_0-1,\beta_0}[0,T]$. Now the second approximation as follows
 \begin{align*}
x^2_0(t)&=\Psi_0(t)-I_{0+}^{\beta_0,\phi}\sum_{i=1}^m d_i(t)\,^{C}D_{0+}^{\beta_i,\phi}\Psi_0(t)+I_{0+}^{\beta_0,\phi}\sum_{i=1}^m d_i(t)I_{0+}^{\beta_0-\beta_i,\phi}\sum_{i=\varkappa_0}^{m}d_i(t)D_{0+}^{\beta_i,\phi}\Psi_0(t) \\
&=\Psi_0(t)-I_{0+}^{\beta_0,\phi}\sum_{i=\varkappa_0}^m d_m(t)D_{0+}^{\beta_i,\phi}\Psi_0(t)+I_{0+}^{\beta_0,\phi}\sum_{i=1}^m d_i(t)I_{0+}^{\beta_0-\beta_i,\phi}\sum_{i=\varkappa_0}^{m}d_i(t)D_{0+}^{\beta_i,\phi}\Psi_0(t) \\
&=\Psi_0(t)+\sum_{k=0}^{1}(-1)^{k+1}I_{0+}^{\beta_0,\phi}\left(\sum_{i=1}^m d_i(t)I_{0+}^{\beta_0-\beta_i,\phi}\right)^k \sum_{i=\varkappa_0}^{m}d_i(t)D_{0+}^{\beta_i,\phi}\Psi_0(t).
 \end{align*}  
Since $d_i(t)D_{0+}^{\beta_i,\phi}\Psi_0(t)\in C[0,T]$ for any $i=\varkappa_0,\ldots,m$ we have $x_0^2(t)\in C^{n_0-1,\beta_0}[0,T]$. Using the induction process it can be proved that the $nth$ approximation of $x_0(t)$ is given by 
\[x^n_0(t)=\Psi_0(t)+\sum_{k=0}^{n-1}(-1)^{k+1}I_{0+}^{\beta_0,\phi}\left(\sum_{i=1}^m d_i(t)I_{0+}^{\beta_0-\beta_i,\phi}\right)^k \sum_{i=\varkappa_0}^{m}d_i(t)D_{0+}^{\beta_i,\phi}\Psi_0(t),\]
where $x_0^n(t)\in C^{n_0-1,\beta_0}[0,T]$. So the sequence converges in $C^{n_0-1,\beta_0}[0,T]$ and the first term $x_0(t)$ of the canonical set is given by 
\[x_0(t)=\lim_{n\to+\infty}x^n_0(t)=\Psi_0(t)+\sum_{k=0}^{+\infty}(-1)^{k+1}I_{0+}^{\beta_0,\phi}\left(\sum_{i=1}^m d_i(t)I_{0+}^{\beta_0-\beta_i,\phi}\right)^k \sum_{i=\varkappa_0}^{m}d_i(t)D_{0+}^{\beta_i,\phi}\Psi_0(t).\]

Now let us give the $jth$ term $x_j(t)$ of the canonical set for any fixed $j\in\{1,\ldots,n_0-1\}$. From \eqref{eq10eq11} it follows
\[x^1_j (t)=\Psi_j(t)-I_{0+}^{\beta_0,\phi}\sum_{i=1}^{m}d_i(t)\,^{C}D_{0+}^{\beta_i,\phi}x^{0}_j(t)=\Psi_j(t)-I_{0+}^{\beta_0,\phi}\sum_{i=1}^{m}d_i(t)\,^{C}D_{0+}^{\beta_i,\phi}\Psi_j(t).\] 
Similarly as in formula \eqref{formula18} we obtain
\begin{equation*}
^{C}D_{0+}^{\beta_i,\phi}\Psi_j(t)=\left\{
\begin{split}
D_{0+}^{\beta_i,\phi}\Psi_j(t),\quad \varkappa_j\leqslant i\leqslant m, \\
0,\quad 1\leqslant i<\varkappa_j.
\end{split}
\right.
\end{equation*}
Hence
\[x^1_j (t)=\Psi_j(t)-I_{0+}^{\beta_0,\phi}\sum_{i=\varkappa_j}^{m}d_i(t)D_{0+}^{\beta_i,\phi}\Psi_j(t),\] 
and $x_j^1(t)\in C^{n_0-1,\beta_0}[0,T]$. A direct computation gives  the second approximation of $x_j(t)$: 
\begin{align*}
x^2_j (t)&=\Psi_j(t)-I_{0+}^{\beta_0,\phi}\sum_{i=1}^{m}d_i(t)\,^{C}D_{0+}^{\beta_i,\phi}x^{1}_j(t) \\
&=\Psi_j(t)-I_{0+}^{\beta_0,\phi}\sum_{i=1}^{m}d_i(t)\,^{C}D_{0+}^{\beta_i,\phi}\left(\Psi_j(t)-I_{0+}^{\beta_0,\phi}\sum_{i=\varkappa_j}^{m}d_i(t)D_{0+}^{\beta_i,\phi}\Psi_j(t)\right) \\
&=\Psi_j(t)-I_{0+}^{\beta_0,\phi}\sum_{i=1}^{m}d_i(t)\,^{C}D_{0+}^{\beta_i,\phi}\left(\Psi_j(t)-I_{0+}^{\beta_0,\phi}\sum_{i=\varkappa_j}^{m}d_i(t)D_{0+}^{\beta_i,\phi}\Psi_j(t)\right) \\
&=\Psi_j(t)-I_{0+}^{\beta_0,\phi}\sum_{i=\varkappa_j}^{m}d_i(t)D_{0+}^{\beta_i,\phi}\Psi_j(t)+I_{0+}^{\beta_0,\phi}\sum_{i=1}^{m}d_i(t)I_{0+}^{\beta_0-\beta_i,\phi}\sum_{i=\varkappa_j}^{m}d_i(t)D_{0+}^{\beta_i,\phi}\Psi_j(t) \\
&=\Psi_j(t)+\sum_{k=0}^{1} (-1)^{k+1} I_{0+}^{\beta_0,\phi}\left(\sum_{i=1}^{m}d_i(t) I_{0+}^{\beta_0-\beta_i,\phi}\right)^k \sum_{i=\varkappa_j}^{m}d_i(t)D_{0+}^{\beta_i,\phi}\Psi_j(t),
\end{align*}
where $x^2_j(t)\in C^{n_0-1,\beta_0}[0,T]$. By using the induction process, the $nth$ approximation of $x_j(t)$ can be given by the formula
\[x^n_j(t)=\Psi_j(t)+\sum_{k=0}^{n-1} (-1)^{k+1} I_{0+}^{\beta_0,\phi}\left(\sum_{i=1}^{m}d_i(t) I_{0+}^{\beta_0-\beta_i,\phi}\right)^k \sum_{i=\varkappa_j}^{m}d_i(t)D_{0+}^{\beta_i,\phi}\Psi_j(t).\]
Thus, we arrive at
\[x_j(t)=\lim_{n\to+\infty}x^n_j(t)=\Psi_j(t)+\sum_{k=0}^{+\infty} (-1)^{k+1} I_{0+}^{\beta_0,\phi}\left(\sum_{i=1}^{m}d_i(t) I_{0+}^{\beta_0-\beta_i,\phi}\right)^k \sum_{i=\varkappa_j}^{m}d_i(t)D_{0+}^{\beta_i,\phi}\Psi_j(t).\]
\end{proof}
\begin{lem}\label{lem3.5}
Let $n_0>n_1$, $d_i\in C[0,T]$ for $i=1,\ldots,m$,  $n_0\geqslant 2$. Assume that there exists $j_0\in\{0,1,\ldots,n_0-2\}$ such that $\mathbb{K}_{j_0}=\emptyset$ and $\mathbb{K}_{j_0+1}\neq\emptyset$. Then there exists a unique canonical set  $x_j\in C^{n_0-1,\beta_0}[0,T]$, $j=0,1,\ldots,n_0-1$, of solutions of equation \eqref{eq3} in the form
\begin{equation}\label{form23}
x_j(t)=\Psi_j(t),\quad j=0,1,\ldots,j_0,
\end{equation}
\begin{equation}\label{form24}
x_j(t)=\Psi_j(t)+\sum_{k=0}^{+\infty} (-1)^{k+1} I_{0+}^{\beta_0,\phi}\left(\sum_{i=1}^{m}d_i(t) I_{0+}^{\beta_0-\beta_i,\phi}\right)^k \sum_{i=\varkappa_j}^{m}d_i(t)D_{0+}^{\beta_i,\phi}\Psi_j(t),
\end{equation} 
for  $j=j_0+1,\ldots,n_1-1$, and 
\begin{equation}\label{form25}
x_j(t)=\Psi_j(t)+\sum_{k=0}^{+\infty} (-1)^{k+1} I_{0+}^{\beta_0,\phi}\left(\sum_{i=1}^{m}d_i(t) I_{0+}^{\beta_0-\beta_i,\phi}\right)^k \sum_{i=1}^{m}d_i(t)D_{0+}^{\beta_i,\phi}\Psi_j(t),
\end{equation} 
for $j=n_1,n_1+1,\ldots,n_0-1$, where $\Psi_j(t)=\frac{(\phi(t)-\phi(0))^{j}}{\Gamma(j+1)}$.
\end{lem}
\begin{proof}
Following the same steps of the proofs of the above lemmas, it can be proved. Indeed, the first approximation of $x_j(t)$ is given by
\[x_j^1(t)=\Psi_j(t),\quad j=0,1,\ldots,j_0,\]
\[x_j^1(t)=\Psi_j(t)-I_{0+}^{\beta_0,\phi} \sum_{i=\varkappa_j}^{m}d_i(t)D_{0+}^{\beta_i,\phi}\Psi_j(t),\quad j=j_0+1,\ldots,n_1-1,\]
\[x_j^1(t)=\Psi_j(t)-I_{0+}^{\beta_0,\phi}\sum_{i=1}^{m}d_i(t)D_{0+}^{\beta_i,\phi}\Psi_j(t),\quad j=n_1,n_1+1,\ldots,n_0-1.\]
Besides the second approximation of $x_j^2(t)$ is given by
\begin{equation*}
x_j^2(t)=\Psi_j(t),\quad j=0,1,\ldots,j_0,
\end{equation*}
\begin{equation*}
x_j^2(t)=\Psi_j(t)+\sum_{k=0}^{1} (-1)^{k+1} I_{0+}^{\beta_0,\phi}\left(\sum_{i=1}^{m}d_i(t) I_{0+}^{\beta_0-\beta_i,\phi}\right)^k \sum_{i=\varkappa_j}^{m}d_i(t)D_{0+}^{\beta_i,\phi}\Psi_j(t),
\end{equation*} 
for $j=j_0+1,\ldots,n_1-1$, and 
\begin{equation*}
x_j^2(t)=\Psi_j(t)+\sum_{k=0}^{1} (-1)^{k+1} I_{0+}^{\beta_0,\phi}\left(\sum_{i=1}^{m}d_i(t) I_{0+}^{\beta_0-\beta_i,\phi}\right)^k \sum_{i=1}^{m}d_i(t)D_{0+}^{\beta_i,\phi}\Psi_j(t),
\end{equation*} 
for $j=n_1,n_1+1,\ldots,n_0-1$. This procedure gives  the $nth$ approximation of $x_j(t)$ and then letting $n\to+\infty$ we arrive at the canonical set of solutions given by \eqref{form23}, \eqref{form24} and \eqref{form25}. It is also clear that $x_j\in C^{n_0-1,\beta}[0,T]$ for any $j=0,1,\ldots,n_0-1$.
\end{proof}

\begin{lem}\label{lem3.4}
Let $n_0=n_1$, $d_i\in C[0,T]$ for $i=1,\ldots,m$,  $n_0\geqslant 2$. Assume that there exists a $j_0\in\{0,1,\ldots,n_0-2\}$ such that $\mathbb{K}_{j_0}=\emptyset$ and $\mathbb{K}_{j_0+1}\neq\emptyset$. Then there exists a unique canonical set $x_j(t)\in C^{n_0-1,\beta_0}[0,T]$, $j=0,1,\ldots,n_0-1$, of solutions of equation \eqref{eq3} in the form
\begin{equation}\label{form21}
x_j(t)=\Psi_j(t),\quad j=0,1,\ldots,j_0,
\end{equation}
\begin{equation}\label{form22}
x_j(t)=\Psi_j(t)+\sum_{k=0}^{+\infty} (-1)^{k+1} I_{0+}^{\beta_0,\phi}\left(\sum_{i=1}^{m}d_i(t) I_{0+}^{\beta_0-\beta_i,\phi}\right)^k \sum_{i=\varkappa_j}^{m}d_i(t)D_{0+}^{\beta_i,\phi}\Psi_j(t),
\end{equation} 
for $j=j_0+1,\ldots,n_0-1$, where $\Psi_j(t)=\frac{(\phi(t)-\phi(0))^{j}}{\Gamma(j+1)}$.
\end{lem}
\begin{proof}
By using a suitable approximation sequence for each case we can get the desired result. In fact, the first approximation of $x_j(t)$ is given by
\[x_j^1(t)=\Psi_j(t),\quad j=0,1,\ldots,j_0,\]
\[x_j^1(t)=\Psi_j(t)-I_{0+}^{\beta_0,\phi} \sum_{i=\varkappa_j}^{m}d_i(t)D_{0+}^{\beta_i,\phi}\Psi_j(t),\quad j=j_0+1,\ldots,n_0-1.\]
The second approximation of $x_j^2(t)$ is given by
\begin{equation*}
x_j^2(t)=\Psi_j(t),\quad j=0,1,\ldots,j_0,
\end{equation*}
\begin{equation*}
x_j^2(t)=\Psi_j(t)+\sum_{k=0}^{1} (-1)^{k+1} I_{0+}^{\beta_0,\phi}\left(\sum_{i=1}^{m}d_i(t) I_{0+}^{\beta_0-\beta_i,\phi}\right)^k \sum_{i=\varkappa_j}^{m}d_i(t)D_{0+}^{\beta_i,\phi}\Psi_j(t),
\end{equation*} 
for $j=j_0+1,\ldots,n_0-1$. Continuing this procedure we obtain the $nth$ approximation of $x_j(t)$ and then letting $n\to+\infty$ we arrive at the canonical set of solutions given by \eqref{form21} and \eqref{form22}. It also follows that $x_j\in C^{n_0-1,\beta_0}[0,T]$ for $j=0,1,\ldots,n_0-1$.
\end{proof}

\begin{lem}\label{lem3.6}
Let $d_i\in C[0,T]$ for  $i=1,\ldots,m$ and $\mathbb{K}_{n_0-1}=\emptyset$. Then there exists a unique canonical set $x_j\in C^{\infty}[0,T]$, $j=0,1,\ldots,n_0-1$, of solutions of equation \eqref{eq3} in the form
\begin{equation}\label{form26}
x_j(t)=\Psi_j(t),\quad j=0,1,\ldots,n_0-1,
\end{equation} 
where $\Psi_j(t)=\frac{(\phi(t)-\phi(0))^{j}}{\Gamma(j+1)}$.
\end{lem}
\begin{proof}
By \eqref{eq10eq11} we have that $x_j^0(t)=\Psi_j(t)$ for $j=0,1,\ldots,n_0-1$. Also it follows that
\[x^1_j (t)=\Psi_j(t)-I_{0+}^{\beta_0,\phi}\sum_{i=1}^{m}d_i(t)\,^{C}D_{0+}^{\beta_i,\phi}\Psi_j(t).\] 
Since $\mathbb{K}_{n_0-1}=\emptyset$ we have $n_0-1<\rm{Re}(\beta_{\it i})\leqslant{\it n}_0$ for $i=1,\ldots,m$. Hence
\[^{C}D_{0+}^{\beta_i,\phi}\Psi_j(t)=0,\]
for any $j=1,\ldots,n_0-1,$ and $x_j^1(t)=\Psi_j(t)$. This procedure gives $x_j^n(t)=\Psi_j(t)$ for any $n\in\mathbb{N}$ and $j=0,1,\ldots,n_0-1$, as well as 
\[x_j(t)=\Psi_j(t)\in C^{\infty}[0,T],\quad j=0,1,\ldots,n_0-1.\] 
\end{proof}

\subsection{Representation of the solutions in the general case}

\begin{thm}\label{thm3.1}
Let $h,d_i\in C[0,T]$, $i=1,\ldots,m$. Then the initial value problem \eqref{eq1} and \eqref{eq4} has a unique solution $x\in C^{n_0-1,\beta_0}[0,T]$ and it is given by the formula
\begin{equation}\label{for27}
x(t)=\sum_{k=0}^{+\infty}(-1)^k I_{0+}^{\beta_0,\phi}\left(\sum_{i=1}^{m}d_i(t)I_{0+}^{\beta_0-\beta_i,\phi}\right)^k h(t),
\end{equation}
whenever $\displaystyle{\sum_{i=1}^m \|d_i\|_{\max}I_{0+}^{\beta_0-\beta_i,\phi}e^{\nu t}}\leqslant Ce^{\nu t}$ for some $\nu>0$, where the constant  $0<C<1$ does not depend on $t$.
\end{thm}
\begin{proof}
Lemma \ref{lem3.1} implies the existence and uniqueness of the solution of the initial value problem \eqref{eq1} and \eqref{eq4}. Let us now find the explicit solution. From \eqref{eq5eq6} the first approximation solution is given by 
\[x^1(t)=I_{0+}^{\beta_0,\phi}h(t)-I_{0+}^{\beta_0,\phi}\sum_{i=1}^{m}d_i(t)\,^{C}D_{0+}^{\beta_i,\phi}I_{0+}^{\beta_0,\phi}h(t).\]   
Lemma \ref{importantpro} implies that $^{C}D_{0+}^{\beta_i,\phi}I_{0+}^{\beta_0,\phi}h(t)=I_{0+}^{\beta_0-\beta_i,\phi}h(t)$. Thus, we have 
\[x^1(t)=I_{0+}^{\beta_0,\phi}h(t)-I_{0+}^{\beta_0,\phi}\sum_{i=1}^{m}d_i(t)I_{0+}^{\beta_0-\beta_i,\phi}h(t)=\sum_{k=0}^{1}(-1)^k I_{0+}^{\beta_0,\phi}\left(\sum_{i=1}^{m}d_i(t)I_{0+}^{\beta_0-\beta_i,\phi}\right)^k h(t),\] 
and $x^1(t)\in C^{n_0-1,\beta_0}[0,T]$. Let us find the second approximation solution. We take $n=2$ in \eqref{eq5eq6}: 
\begin{align*}
x^2(t)&=I_{0+}^{\beta_0,\phi}h(t)-I_{0+}^{\beta_0,\phi}\sum_{i=1}^{m}d_i(t)\,^{C}D_{0+}^{\beta_i,\phi}x^{1}(t) \\
&=I_{0+}^{\beta_0,\phi}h(t)-I_{0+}^{\beta_0,\phi}\sum_{i=1}^{m}d_i(t)\,^{C}D_{0+}^{\beta_i,\phi}\left(I_{0+}^{\beta_0,\phi}h(t)-I_{0+}^{\beta_0,\phi}\sum_{i=1}^{m}d_i(t)I_{0+}^{\beta_0-\beta_i,\phi}h(t)\right) \\
&=I_{0+}^{\beta_0,\phi}h(t)-I_{0+}^{\beta_0,\phi}\sum_{i=1}^{m}d_i(t)I_{0+}^{\beta_0-\beta_i,\phi}h(t)+I_{0+}^{\beta_0,\phi}\left(\sum_{i=1}^{m}d_i(t) I_{0+}^{\beta_0-\beta_i,\phi}\right)^2 h(t) \\
&=\sum_{k=0}^2 (-1)^k I_{0+}^{\beta_0,\phi}\left(\sum_{i=1}^{m}d_i(t) I_{0+}^{\beta_0-\beta_i,\phi}\right)^k h(t).
\end{align*}   
Now we assume that for $n\in\mathbb{N}$ the approximation solution is given by 
\[x^n(t)=\sum_{k=0}^n(-1)^k I_{0+}^{\beta_0,\phi}\left(\sum_{i=1}^{m}d_i(t) I_{0+}^{\beta_0-\beta_i,\phi}\right)^k h(t),\]
and $x^n(t)\in C^{n_0-1,\beta_0}[0,T]$. Let us show that  the approximation solution holds for $n+1$. Indeed, we have
\begin{align*}
x^{n+1}(t)=&=I_{0+}^{\beta_0,\phi}h(t)-I_{0+}^{\alpha_0}\sum_{i=1}^{m}d_i(t)\,^{C}D_{0+}^{\beta_i,\phi}y^{n}(t) \\
&=I_{0+}^{\beta_0,\phi}h(t)-I_{0+}^{\beta_0,\phi}\sum_{i=1}^{m}d_i(t)\,^{C}D_{0+}^{\beta_i,\phi}\left(\sum_{k=0}^n (-1)^k I_{0+}^{\beta_0,\phi}\left(\sum_{i=1}^{m}d_i(t) I_{0+}^{\beta_0-\beta_i,\phi}\right)^k h(t)\right) \\
&=I_{0+}^{\beta_0,\phi}h(t)+\sum_{k=0}^{n}(-1)^{k+1} I_{0+}^{\beta_0,\phi}\left(\sum_{i=1}^{m}d_i(t) I_{0+}^{\beta_0-\beta_i,\phi}\right)^{k+1}h(t) \\
&=\sum_{k=0}^{n+1} (-1)^k I_{0+}^{\beta_0,\phi}\left(\sum_{i=1}^{m}d_i(t) I_{0+}^{\beta_0-\beta_i,\phi}\right)^k h(t).
\end{align*}
Finally, we obtain that the $nth$ approximation solution is given by 
\[x^n(t)=\sum_{k=0}^{n}(-1)^k I_{0+}^{\beta_0,\phi}\left(\sum_{i=1}^{m}d_i(t) I_{0+}^{\beta_0-\beta_i,\phi}\right)^k h(t)\]
and 
\[x(t)=\lim_{n\to+\infty}y^n (t)=\sum_{k=0}^{+\infty}(-1)^k I_{0+}^{\beta_0,\phi}\left(\sum_{i=1}^{m}d_i(t) I_{0+}^{\beta_0-\beta_i,\phi}\right)^k h(t).\]
\end{proof}
Now we state some consequences of Theorem \ref{thm3.1} without proofs. Their proofs follow from Theorem \ref{thm3.1} and the superposition principle.
\begin{thm}\label{thm3.2}
Let $n_0=n_1$, $\beta_{m}=0$ and $h,d_i\in C[0,T]$, $i=1,\ldots,m$. Then the initial value problem \eqref{eq1} and \eqref{eq2} has a unique solution $x\in C^{n_0-1,\beta_0}[0,T]$, which is given by
\begin{equation}\label{for27}
x(t)=\sum_{j=0}^{n_0-1}c_jx_j(t)+\sum_{k=0}^{+\infty}(-1)^k I_{0+}^{\beta_0,\phi}\left(\sum_{i=1}^{m}d_i(t)I_{0+}^{\beta_0-\beta_i,\phi}\right)^k h(t),
\end{equation}
whenever $\displaystyle{\sum_{i=1}^m \|d_i\|_{\max}I_{0+}^{\beta_0-\beta_i,\phi}e^{\nu t}}\leqslant Ce^{\nu t}$ for some $\nu>0$, where the constant  $0<C<1$ does not depend on $t$ and $x_j(t)$ is the canonical set from Lemma \ref{lem3.2}.
\end{thm}
\begin{thm}\label{cor3.2.1}
Let $n_0>n_1$, $\beta_{m}=0$ and $h,d_i\in C[0,T]$, $i=1,\ldots,m$. Then the initial value problem \eqref{eq1} and \eqref{eq2} has a unique solution $x\in C^{n_0-1,\beta_0}[0,T]$, which is given by
\[x(t)=\sum_{j=0}^{n_0-1}c_jx_j(t)+\sum_{k=0}^{+\infty}(-1)^k I_{0+}^{\beta_0,\phi}\left(\sum_{i=1}^{m}d_i(t)I_{0+}^{\beta_0-\beta_i,\phi}\right)^k h(t),\]
whenever $\displaystyle{\sum_{i=1}^m \|d_i\|_{\max}I_{0+}^{\beta_0-\beta_i,\phi}e^{\nu t}}\leqslant Ce^{\nu t}$ for some $\nu>0$, where the constant  $0<C<1$ does not depend on $t$ and $x_j(t)$ is the canonical set from Lemma \ref{lem3.3}.
\end{thm}

\begin{thm}\label{cor3.2.2}
Let $n_0=n_1$, $n_0\geqslant 2$ and there exists a $j_0\in\{0,1,\ldots,n_0-2\}$ such that $\mathbb{K}_{j_0}=\emptyset$ and $\mathbb{K}_{j_0+1}\neq\emptyset$, and $h,d_i\in C[0,T]$, $i=1,\ldots,m$. Then the initial value problem \eqref{eq1} and \eqref{eq2} has a unique solution $x\in C^{n_0-1,\beta}[0,T]$, which is given by
\[x(t)=\sum_{j=0}^{n_0-1}c_jx_j(t)+\sum_{k=0}^{+\infty}(-1)^k I_{0+}^{\beta_0,\phi}\left(\sum_{i=1}^{m}d_i(t)I_{0+}^{\beta_0-\beta_i,\phi}\right)^k h(t),\]
whenever $\displaystyle{\sum_{i=1}^m \|d_i\|_{\max}I_{0+}^{\beta_0-\beta_i,\phi}e^{\nu t}}\leqslant Ce^{\nu t}$ for some $\nu>0$, where the constant  $0<C<1$ does not depend on $t$ and $x_j(t)$ is the canonical set from Lemma \ref{lem3.4}.
\end{thm}

\begin{thm}\label{cor3.2.3}
Let $n_0>n_1$, $n_0\geqslant 2$ and there exists a $j_0\in\{0,1,\ldots,n_0-2\}$ such that $\mathbb{K}_{j_0}=\emptyset$ and $\mathbb{K}_{j_0+1}\neq\emptyset$, and $h,d_i\in C[0,T]$, $i=1,\ldots,m$. Then the initial value problem \eqref{eq1} and \eqref{eq2} has a unique solution $x\in C^{n_0-1,\beta_0}[0,T]$, which is given by
\[x(t)=\sum_{j=0}^{n_0-1}c_jx_j(t)+\sum_{k=0}^{+\infty}(-1)^k I_{0+}^{\beta_0,\phi}\left(\sum_{i=1}^{m}d_i(t)I_{0+}^{\beta_0-\beta_i,\phi}\right)^k h(t),\]
whenever $\displaystyle{\sum_{i=1}^m \|d_i\|_{\max}I_{0+}^{\beta_0-\beta_i,\phi}e^{\nu t}}\leqslant Ce^{\nu t}$ for some $\nu>0$, where the constant  $0<C<1$ does not depend on $t$ and $x_j(t)$ is the canonical set from Lemma \ref{lem3.5}.
\end{thm}

\begin{thm}\label{cor3.2.4}
Let $n_0=n_1$, $\mathbb{K}_{n_0-1}=\emptyset$ and $h,d_i\in C[0,T]$, $i=1,\ldots,m$. Then the initial value problem \eqref{eq1} and \eqref{eq2} has a unique solution $x\in C^{n_0-1,\beta_0}[0,T]$, which is given by
\[x(t)=\sum_{j=0}^{n_0-1}c_jx_j(t)+\sum_{k=0}^{+\infty}(-1)^k I_{0+}^{\beta_0,\phi}\left(\sum_{i=1}^{m}d_i(t)I_{0+}^{\beta_0-\beta_i,\phi}\right)^k h(t),\]
whenever $\displaystyle{\sum_{i=1}^m \|d_i\|_{\max}I_{0+}^{\beta_0-\beta_i,\phi}e^{\nu t}}\leqslant Ce^{\nu t}$ for some $\nu>0$, where the constant  $0<C<1$ does not depend on $t$ and $x_j(t)$ is the canonical set from Lemma \ref{lem3.6}.
\end{thm}

\subsection{Examples}\label{examples}
We use the notations $I_{0+}^{\beta}$, $^{C}D_{0+}^{\beta}$ instead of $I_{0+}^{\beta,\phi}$, $^{C}D_{0+}^{\beta,\phi}$ when $\phi(t)=t$. For the further calculations we need the following expression:
\begin{equation}\label{estimate}
I_{0+}^{\lambda}(t^{\beta})=\frac{\Gamma(\beta+1)}{\Gamma(\lambda+\beta+1)}t^{\lambda+\beta},\quad t>0,\,\,\lambda>0,\,\,\beta>-1.
\end{equation}
Let us now consider the (fractional) initial problem:
\begin{equation}\label{example1}
^{C}D^{\beta_0}_{0+}x(t)+t^{\alpha}\,^{C}D^{\beta_1}_{0+}x(t)=t^{\beta},\quad x(t)|_{_{_{t=0+}}}=0,
\end{equation}
where $0<\beta_1<\beta_0<1$ and $\alpha,\beta\in\mathbb{R}_+$. Remark \ref{newcond} and Theorem \ref{thm3.1} imply that the solution of equation \eqref{example1} is given by   
\[x(t)=\sum_{k=0}^{+\infty}(-1)^k I_{0+}^{\beta_0}\left(t^{\alpha}I_{0+}^{\beta_0-\beta_1}\right)^k t^{\beta}.\] 
By using \eqref{estimate} it can be proved that
\[\left(t^{\alpha}I_{0+}^{\beta_0-\beta_1}\right)^{k}t^{\beta}=\prod_{j=0}^{k-1}\frac{\Gamma(j(\alpha+\beta_0-\beta_1)+\beta+1)}{\Gamma(j(\alpha+\beta_0-\beta_1)+\beta_0-\beta_1+\beta+1)}t^{k(\alpha+\beta_0-\beta_1)+\beta}.\]
Thus, we have 
\[x(t)=\sum_{k=0}^{+\infty}(-1)^k I_{0+}^{\beta_0}\left(t^{\alpha}I_{0+}^{\beta_0-\beta_1}\right)^k t^{\beta}=-I_{0+}^{\beta_0}\big(t^{\beta}E^{\beta_0-\beta_1}_{1,\alpha+\beta_0-\beta_1,\beta}((-t)^{\alpha+\beta_0-\beta_1})\big),\]
where 
\[E^{\lambda}_{\alpha,\beta,\gamma}=\sum_{k=0}^{+\infty}c_k z^k,\quad z\in\mathbb{C},\]
with 
\[c_0=1,\quad c_k=\prod_{j=0}^{k-1}\frac{\Gamma(\alpha[j\beta+\gamma]+1)}{\Gamma(\alpha[j\beta+\gamma]+\lambda+1)},\quad k=1,2,\ldots,\quad \alpha,\beta,\lambda\in\mathbb{R},\,\gamma\in\mathbb{C}.\]
Notice that in the case $\lambda=\alpha$ we obtain that $E^{\alpha}_{\alpha,\beta,\gamma}$ is the generalized (Kilbas--Saigo) Mittag--Leffler type function \cite[Chapter 5]{mittagbook}. 

\medskip Consider the following fractional equation:
\begin{equation}\label{example2}
^{C}D^{\beta_0}_{0+}x(t)+t^{\alpha}\,^{C}D^{\beta_1}_{0+}x(t)=t^{\beta},\quad x(t)|_{_{_{t=0+}}}=0,\quad x^{\prime}(t)|_{_{_{t=0+}}}=0
\end{equation}
where $0<\beta_1<1<\beta_0<2$, $\beta_0-\beta_1=1$ and $\alpha,\beta\in\mathbb{R}_+$. By Theorem \ref{thm3.1} and previous example we obtain the solution as follows   
\[x(t)=-I_{0+}^{\beta_0}\big(t^{\beta}E_{1,\alpha+\beta_0-\beta_1,\beta}((-t)^{\alpha+\beta_0-\beta_1})\big),\]
for the the generalized (Kilbas--Saigo) Mittag--Leffler type function $E_{1,\alpha+\beta_0-\beta_1,\beta}(z)$. Moreover, by \cite[Theorem 5.27]{mittagbook} and $\beta=\beta_0$ we get
\[I_{0+}^{\beta_0}\big(t^{\beta_0}E_{1,\alpha+\beta_0-\beta_1,\beta_0}((-t)^{\alpha+\beta_0-\beta_1})\big)=-t^{\beta_1-\alpha+1}\big(E_{1,\alpha+\beta_0-\beta_1,\beta_0}((-t)^{\alpha+\beta_0-\beta_1})-1\big).\]
Hence the solution of equation \eqref{example2} is given by
\[x(t)=-\frac{1}{\Gamma(\beta_0)}\int_0^t (t-s)^{\beta_0-1}s^{\beta}E_{1,\alpha+\beta_0-\beta_1,\beta}((-s)^{\alpha+\beta_0-\beta_1})ds.\]
While for $\beta=\beta_0$ we have 
\[x(t)=t^{\beta_1-\alpha+1}E_{1,\alpha+\beta_0-\beta_1,\beta_0}((-t)^{\alpha+\beta_0-\beta_1})-t^{\beta_1-\alpha+1}.\]

\section{Constant coefficients}\label{addition}

In this section we treat the case of constant coefficients. We consider \eqref{eq1} with $d_i(t)=\lambda_i\in\mathbb{C}$ for any  $i=0,1,\ldots,m$. In fact
\begin{equation}\label{eq1constant}
^{C}D_{0+}^{\beta_0,\phi}x(t)+\sum_{i=1}^{m}\lambda_i\,^{C}D_{0+}^{\beta_i,\phi}x(t)=h(t),\quad t\in[0,T],\; m\in\mathbb{N},
\end{equation}
under the initial conditions 
\begin{equation}\label{eq2constant}
\left(\frac1{\phi^{\prime}(t)}\frac{d}{dt}\right)^k x(t)|_{_{t=+0}}=c_k\in\mathbb{R},\quad k=0,1,\ldots,n_0-1, 
\end{equation}
or 
\begin{equation}\label{eq4constant}
\left(\frac1{\phi^{\prime}(t)}\frac{d}{dt}\right)^k x(t)|_{_{t=+0}}=0,\quad k=0,1,\ldots,n_0-1,
\end{equation}
where $\beta_i\in\mathbb{C}$, $\rm{Re}(\beta_{\it i})>0$, $i=0,1,\ldots,m-1$, $\rm{Re}(\beta_0)>\rm{Re}(\beta_1)>\ldots>\rm{Re}(\beta_{\it m})\geqslant0$ (If $\rm{Re}(\beta_{\it m})=0$, then we assume $\rm{Im}(\beta_{\it m})=0$ as well) and $n_i$ are non-negative integers satisfying $n_i-1<\rm{Re}(\beta_{\it i})<{\it n_i},$ $n_i=\lfloor \rm{Re}\beta_{\it i}\rfloor+1$ (or $n_i=-\lfloor-\rm{Re}(\beta_{\it i})\rfloor$), $i=0,1,\ldots,m$.

Let us recall two useful expressions, see e.g. \cite[Properties 2.18, 2.20]{kilbas}.
\begin{equation}\label{proi}
I_{0+}^{\alpha,\phi}\big(\phi(t)-\phi(0)\big)^{\beta}=\frac{\Gamma(\beta+1)}{\Gamma(\alpha+\beta+1)}(\phi(t)-\phi(0))^{\alpha+\beta},\quad \rm{Re}(\alpha)>0,\,\rm{Re}(\beta)>0.
\end{equation}
\begin{equation}\label{prod}
D_{0+}^{\alpha,\phi}\big(\phi(t)-\phi(0)\big)^{\beta}=\frac{\Gamma(\beta+1)}{\Gamma(\beta+1-\alpha)}(\phi(t)-\phi(0))^{\beta-\alpha},\quad \rm{Re}(\alpha)>0,\,\rm{Re}(\beta)>0.
\end{equation}
We recall the definition of the multivariate Mittag-Leffler function. We use it in the representations of solutions of the considered fractional differential equations with constants coefficients.
\begin{defn}
The multivariate Mittag-Leffler function $E_{(a_1,\ldots,a_n),b}(z_1,\ldots,z_n)$ of $n$ complex variables $z_1,\ldots,z_n\in\mathbb{C}$ with complex parameters $a_1,\ldots,a_n,b\in\mathbb{C}$ (with positive real parts) is defined by
\begin{equation}\label{mmittag}
E_{(a_1,\ldots,a_n),b}(z_1,\ldots,z_n)=\sum_{k=0}^{+\infty}\sum_{l_1+\cdots+l_n= k,\,\, l_1,\ldots,l_n\geq0}{{k}\choose{l_1,\ldots,l_n}}\frac{\displaystyle{\prod_{i=1}^n z_i^{l_i}}}{\Gamma\left(b+\displaystyle{\sum_{i=1}^n a_i l_i}\right)},
\end{equation}  
where the multinomial coefficients are given by 
\[{{k}\choose{l_1,\ldots,l_n}}=\frac{k!}{l_1!\times\cdots\times l_n!}.\]
\end{defn}
Notice that from the viewpoint of fractional calculus and applications, the readers may find  interesting more detailed studies of the bivariate $(n=2)$ \cite{bivariate} and trivariate $(n=3)$ \cite{threevariate} cases.

Below we show that the representation of the solution of equation \eqref{eq1constant} and \eqref{eq2constant} is given by the multivariate Mittag-Leffler function. 
\begin{thm}\label{cor3.2.1constant}
Let $n_0>n_1$, $\beta_{m}=0$ and $h\in C[0,T]$.  Then the initial value problem \eqref{eq1constant} and \eqref{eq2constant} has a unique solution $x\in C^{n_0-1,\beta_0}[0,T]$ and it is given by
\begin{align*}
x(t)=&\sum_{j=0}^{n_0-1}c_j x_j(t)+\int_0^t \phi^{\prime}(s)(\phi(t)-\phi(s))^{\beta_0-1}\times \\
&\times E_{(\beta_0-\beta_1,\ldots,\beta_0-\beta_m),\beta_0}(-\lambda_1 (\phi(t)-\phi(s))^{\beta_0-\beta_1},\cdots,-\lambda_m (\phi(t)-\phi(s))^{\beta_0-\beta_m})h(s)ds,
\end{align*}
where 
\begin{align*}
x_j(t)=&\Psi_j(t)+\sum_{i=\varkappa_j}^m \lambda_i (\phi(t)-\phi(0))^{j+\beta_0-\beta_i}\times \\
&\times E_{(\beta_0-\beta_1,\ldots,\beta_0-\beta_m),j+1+\beta_0-\beta_i}(\lambda_1 (\phi(t)-\phi(0))^{\beta_0-\beta_1},\cdots,\lambda_m (\phi(t)-\phi(0))^{\beta_0-\beta_m}),
\end{align*}
for  $j=0,1,\ldots,n_1-1$, where $\varkappa_j=\displaystyle{\min\{\mathbb{K}_j\}}$, and 
\begin{align*}
x_j(t)=&\Psi_j(t)+\sum_{i=1}^m \lambda_i (\phi(t)-\phi(0))^{j+\beta_0-\beta_i}\times \\
&\times E_{(\beta_0-\beta_1,\ldots,\beta_0-\beta_m),j+1+\beta_0-\beta_i}(\lambda_1 (\phi(t)-\phi(0))^{\beta_0-\beta_1},\cdots,\lambda_m (\phi(t)-\phi(0))^{\beta_0-\beta_m}),
\end{align*}
for $j=n_1,n_1+1,\ldots,n_0-1$, with $\Psi_j(t)=\frac{(\phi(t)-\phi(0))^{j}}{\Gamma(j+1)}$, whenever $\displaystyle{\sum_{i=1}^m |\lambda_i|I_{0+}^{\beta_0-\beta_i,\phi}e^{\nu t}}\leqslant Ce^{\nu t}$ for some $\nu>0$, where the constant  $0<C<1$ does not depend on $t$.
\end{thm}    
\begin{proof}
By Theorem \ref{cor3.2.1} and Lemma \ref{lem3.3}, the solution of \eqref{eq1constant} is given by
\begin{equation}\label{secondterm}
x(t)=\sum_{j=0}^{n_0-1}c_jx_j(t)+\sum_{k=0}^{+\infty}(-1)^k I_{0+}^{\beta_0,\phi}\left(\sum_{i=1}^{m}\lambda_iI_{0+}^{\beta_0-\beta_i,\phi}\right)^k h(t),
\end{equation}
where 
\begin{equation}\label{coeconstant}
x_j(t)=\Psi_j(t)+\sum_{k=0}^{+\infty} (-1)^{k+1} I_{0+}^{\beta_0,\phi}\left(\sum_{i=1}^{m}\lambda_i I_{0+}^{\beta_0-\beta_i,\phi}\right)^k \sum_{i=\varkappa_j}^{m}\lambda_iD_{0+}^{\beta_i,\phi}\Psi_j(t),
\end{equation}
for  $j=0,1,\ldots,n_1-1$, where $\varkappa_j=\displaystyle{\min\{\mathbb{K}_j\}}$ and 
\begin{equation}\label{coeconstant2}
x_j(t)=\Psi_j(t)+\sum_{k=0}^{+\infty} (-1)^{k+1} I_{0+}^{\beta_0,\phi}\left(\sum_{i=1}^{m}\lambda_iI_{0+}^{\beta_0-\beta_i,\phi}\right)^k \sum_{i=1}^{m}\lambda_iD_{0+}^{\beta_i,\phi}\Psi_j(t),
\end{equation}
for $j=n_1,n_1+1,\ldots,n_0-1$, with $\Psi_j(t)=\frac{(\phi(t)-\phi(0))^{j}}{\Gamma(j+1)}$.

Let us calculate the second term of \eqref{secondterm}. By the multinomial theorem we get
\begin{align*}
&\sum_{k=0}^{+\infty}(-1)^{k}I_{0+}^{\beta_0,\phi}\left(\sum_{i=1}^{m}\lambda_iI_{0+}^{\beta_0-\beta_i,\phi}\right)^{k}h(t) \\
&=\sum_{k=0}^{+\infty}(-1)^{k}I_{0+}^{\beta_0,\phi}\left(\sum_{l_1+\cdots+l_m=k}{{k}\choose{l_1,\ldots,l_m}}\prod_{i=1}^{m}\lambda_i^{l_i}I_{0+}^{l_i (\beta_0-\beta_i),\phi}\right)h(t) \\
&=\sum_{k=0}^{+\infty}(-1)^{k}\left(\sum_{l_1+\cdots+l_m=k}{{k}\choose{l_1,\ldots,l_m}}\prod_{i=1}^{m}\lambda_i^{l_i}\right)I_{0+}^{\beta_0+\sum_{i=1}^{m}l_i (\beta_0-\beta_i),\phi}h(t) \\
&=\sum_{k=0}^{+\infty}\sum_{l_1+\cdots+l_m=k}{{k}\choose{l_1,\ldots,l_m}}\frac{\displaystyle{\prod_{i=1}^{m}(-\lambda_i)^{l_i}}}{\Gamma\left(\beta_0+\displaystyle{\sum_{i=1}^{m}l_i (\beta_0-\beta_i)}\right)}\times \\
&\hspace{3cm}\times\int_0^t \phi^{\prime}(s)(\phi(t)-\phi(s))^{\beta_0+\sum_{i=1}^{m}l_i (\beta_0-\beta_i)-1}h(s)ds \\
&=\int_0^t \phi^{\prime}(s)(\phi(t)-\phi(s))^{\beta_0-1}\left(\sum_{k=0}^{+\infty}\sum_{l_1+\cdots+l_m=k}{{k}\choose{l_1,\ldots,l_m}}\frac{\prod_{i=1}^{m}(-\lambda_i (\phi(t)-\phi(s))^{\beta_0-\beta_i})^{l_i}}{\Gamma\left(\beta_0+\displaystyle{\sum_{i=1}^{m}l_i (\beta_0-\beta_i)}\right)}\right)h(s)ds \\
&=\int_0^t \phi^{\prime}(s)(\phi(t)-\phi(s))^{\beta_0-1}\times \\
&\hspace{1cm}E_{(\beta_0-\beta_1,\ldots,\beta_0-\beta_m),\beta_0}(-\lambda_1 (\phi(t)-\phi(s))^{\beta_0-\beta_1},\cdots,-\lambda_m (\phi(t)-\phi(s))^{\beta_0-\beta_m})h(s)ds.
\end{align*}
By \eqref{prod} and \eqref{coeconstant2} we also have that 
\begin{align*}
&\sum_{k=0}^{+\infty}(-1)^{k+1}I_{0+}^{\beta_0,\phi}\left(\sum_{i=1}^{m}\lambda_iI_{0+}^{\beta_0-\beta_i,\phi}\right)^{k}\sum_{i=0}^{m}\lambda_i\frac{(\phi(t)-\phi(0))^{j-\beta_i}}{\Gamma(j-\beta_i+1)}\\
&=\sum_{k=0}^{+\infty}(-1)^{k+1}I_{0+}^{\beta_0,\phi}\left(\sum_{l_1+\cdots+l_m=k}{{k}\choose{l_1,\ldots,l_m}}\prod_{i=1}^{m}\lambda_i^{l_i}I_{0+}^{l_i (\beta_0-\beta_i),\phi}\right)\sum_{i=0}^{m}\lambda_i\frac{(\phi(t)-\phi(0))^{j-\beta_i}}{\Gamma(j-\beta_i+1)} \\
&=\sum_{k=0}^{+\infty}(-1)^{k+1}\left(\sum_{l_1+\cdots+l_m=k}{{k}\choose{l_1,\ldots,l_m}}\prod_{i=1}^{m}\lambda_i^{l_i}I_{0+}^{\beta_0+l_i (\beta_0-\beta_i),\phi}\right)\sum_{i=0}^{m}\lambda_i\frac{(\phi(t)-\phi(0))^{j-\beta_i}}{\Gamma(j-\beta_i+1)} \\
&=\sum_{k=0}^{+\infty}(-1)^{k+1}\left(\sum_{l_1+\cdots+l_m=k}{{k}\choose{l_1,\ldots,l_m}}\prod_{i=1}^{m}\lambda_i^{l_i}\right)\sum_{i=0}^{m}\lambda_i\frac{I_{0+}^{\beta_0+\sum_{i=1}^{m}l_i (\beta_0-\beta_i),\phi}\big((\phi(t)-\phi(0))^{j-\beta_i}\big)}{\Gamma(j-\beta_i+1)}. 
\end{align*}
By the estimation \eqref{proi} we have
\begin{align*}
&\sum_{k=0}^{+\infty}(-1)^{k+1}\left(\sum_{l_1+\cdots+l_m=k}{{k}\choose{l_1,\ldots,l_m}}\prod_{i=1}^{m}\lambda_i^{l_i}\right)\sum_{i=0}^{m}\lambda_i\frac{I_{0+}^{\beta_0+\sum_{i=1}^{m}l_i(\beta_0-\beta_i),\phi}\big((\phi(t)-\phi(0))^{j-\beta_i}\big)}{\Gamma(j-\beta_i+1)} \\
&=\sum_{k=0}^{+\infty}(-1)^{k+1}\left(\sum_{l_1+\cdots+l_m=k}{{k}\choose{l_1,\ldots,l_m}}\prod_{i=1}^{m}\lambda_i^{l_i}\right)\sum_{i=0}^{m}\lambda_i\frac{(\phi(t)-\phi(0))^{j-\beta_i+\beta_0+\sum_{i=1}^{m}l_i (\beta_0-\beta_i)}}{\Gamma(j-\beta_i+1+\beta_0+\sum_{i-1}^{m}l_i (\beta_0-\beta_i))} \\
&=\sum_{i=0}^{m}(-\lambda_i)(\phi(t)-\phi(0))^{j-\beta_i+\beta_0}\sum_{k=0}^{+\infty}\left(\sum_{l_1+\cdots+l_m=k}{{k}\choose{l_1,\ldots,l_m}}\frac{\prod_{i=1}^{m}(-\lambda_i (\phi(t)-\phi(0))^{\beta_0-\beta_i})^{l_i}}{\Gamma(j-\beta_i+1+\beta_0+\sum_{i-1}^{m}l_i (\beta_0-\beta_i))}\right).
\end{align*} 
From the definition of the multivariate Mittag-Leffler function we obtain
\begin{align*}
&\sum_{i=0}^{m}(-\lambda_i)(\phi(t)-\phi(0))^{j-\beta_i+\beta_0}\sum_{k=0}^{+\infty}\left(\sum_{l_1+\cdots+l_m=k}{{k}\choose{l_1,\ldots,l_m}}\frac{\prod_{i=1}^{m}(-\lambda_i (\phi(t)-\phi(0))^{\beta_0-\beta_i})^{l_i}}{\Gamma(j-\beta_i+1+\beta_0+\sum_{i-1}^{m}l_i (\beta_0-\beta_i))}\right) \\
&=\sum_{i=0}^{m}(-\lambda_i)(\phi(t)-\phi(0))^{j-\beta_i+\beta_0}\times \\
&\hspace{1cm}\times E_{(\beta_0-\beta_1,\ldots,\beta_0-\beta_m),j+1+\beta_0-\beta_i}(-\lambda_1 (\phi(t)-\phi(0))^{\beta_0-\beta_1},\cdots,-\lambda_n (\phi(t)-\phi(0))^{\beta_0-\beta_m}).
\end{align*} 
Hence, we have 
\begin{align*}
x_j(t)&=\Psi_{j}(t)+\sum_{i=0}^{m}(-\lambda_i)(\phi(t)-\phi(0))^{j+\beta_0-\beta_i}\times \\
&\times E_{(\beta_0-\beta_1,\ldots,\beta_0-\beta_m),j+1+\beta_0-\beta_i}(-\lambda_1 (\phi(t)-\phi(0))^{\beta_0-\beta_1},\cdots,-\lambda_m (\phi(t)-\phi(0))^{\beta_0-\beta_m}),
\end{align*}
for any $j=n_1,\ldots,n_0-1$. The case $j=0,\ldots,n_0-1$ can be  proved similarly.
\end{proof}
Now we state other cases and it can be proved similarly by using the results giving in Section \ref{main}. Notice that the next result extend some of the results given in \cite{luchko}.
\begin{thm}\label{thm3.1con}
Let $h\in C[0,T]$. Then the initial value problem \eqref{eq1constant} and \eqref{eq4constant} has a unique solution $x\in C^{n_0-1,\beta_0}[0,T]$ and it is given by the formula
\begin{align*}
x(t)&=\int_0^t \phi^{\prime}(s)(\phi(t)-\phi(s))^{\beta_0-1}\times \\
&\times E_{(\beta_0-\beta_1,\ldots,\beta_0-\beta_m),\beta_0}(-\lambda_1 (\phi(t)-\phi(s))^{\beta_0-\beta_1},\cdots,-\lambda_m (\phi(t)-\phi(s))^{\beta_0-\beta_m})h(s)ds,
\end{align*}
whenever $\displaystyle{\sum_{i=1}^m |\lambda_i|I_{0+}^{\beta_0-\beta_i,\phi}e^{\nu t}}\leqslant Ce^{\nu t}$ for some $\nu>0$, where the constant  $0<C<1$ does not depend on $t$.
\end{thm}
\begin{thm}\label{thm3.2con}
Let $n_0=n_1$, $\beta_{m}=0$ and $h\in C[0,T]$.  Then the initial value problem \eqref{eq1constant} and \eqref{eq2constant} has a unique solution $x\in C^{n_0-1,\beta_0}[0,T]$, which is given by
\begin{align*}
x(t)=&\sum_{j=0}^{n_0-1}c_j x_j(t)+\int_0^t \phi^{\prime}(s)(\phi(t)-\phi(s))^{\beta_0-1}\times \\
&\times E_{(\beta_0-\beta_1,\ldots,\beta_0-\beta_m),\beta_0}(-\lambda_1 (\phi(t)-\phi(s))^{\beta_0-\beta_1},\cdots,-\lambda_m (\phi(t)-\phi(s))^{\beta_0-\beta_m})h(s)ds,
\end{align*}
where 
\begin{align*}
x_j(t)=&\Psi_j(t)+\sum_{i=\varkappa_j}^m \lambda_i (\phi(t)-\phi(0))^{j+\beta_0-\beta_i}\times \\
&\times E_{(\beta_0-\beta_1,\ldots,\beta_0-\beta_m),j+1+\beta_0-\beta_i}(\lambda_1 (\phi(t)-\phi(0))^{\beta_0-\beta_1},\cdots,\lambda_m (\phi(t)-\phi(0))^{\beta_0-\beta_m}),
\end{align*}
for  $j=0,1,\ldots,n_0-1$, with $\Psi_j(t)=\frac{(\phi(t)-\phi(0))^{j}}{\Gamma(j+1)}$, whenever $\displaystyle{\sum_{i=1}^m |\lambda_i|I_{0+}^{\beta_0-\beta_i,\phi}e^{\nu t}}\leqslant Ce^{\nu t}$ for some $\nu>0$, where the constant  $0<C<1$ does not depend on $t$.
\end{thm}

\begin{thm}\label{cor3.2.2con}
Let $n_0=n_1$, $n_0\geqslant 2$ and there exists a $j_0\in\{0,1,\ldots,n_0-2\}$ such that $\mathbb{K}_{j_0}=\emptyset$ and $\mathbb{K}_{j_0+1}\neq\emptyset$, and $h\in C[0,T]$. Then the initial value problem \eqref{eq1constant} and \eqref{eq2constant} has a unique solution $x\in C^{n_0-1,\beta}[0,T]$, which is given by
\begin{align*}
x(t)=&\sum_{j=0}^{n_0-1}c_j x_j(t)+\int_0^t \phi^{\prime}(s)(\phi(t)-\phi(s))^{\beta_0-1}\times \\
&\times E_{(\beta_0-\beta_1,\ldots,\beta_0-\beta_m),\beta_0}(-\lambda_1 (\phi(t)-\phi(s))^{\beta_0-\beta_1},\cdots,-\lambda_m (\phi(t)-\phi(s))^{\beta_0-\beta_m})h(s)ds,
\end{align*}
where 
\[x_j(t)=\Psi_j(t),\quad j=0,1,\ldots,j_0,\]
and 
\begin{align*}
x_j(t)=&\Psi_j(t)+\sum_{i=\varkappa_j}^m \lambda_i (\phi(t)-\phi(0))^{j+\beta_0-\beta_i}\times \\
&\times E_{(\beta_0-\beta_1,\ldots,\beta_0-\beta_m),j+1+\beta_0-\beta_i}(\lambda_1 (\phi(t)-\phi(0))^{\beta_0-\beta_1},\cdots,\lambda_m (\phi(t)-\phi(0))^{\beta_0-\beta_m}),
\end{align*}
for $j=j_0+1,\ldots,n_0-1$, with $\Psi_j(t)=\frac{(\phi(t)-\phi(0))^{j}}{\Gamma(j+1)}$, where $\varkappa_j=\displaystyle{\min\{\mathbb{K}_j\}}$ and whenever $\displaystyle{\sum_{i=1}^m |\lambda_i|I_{0+}^{\beta_0-\beta_i,\phi}e^{\nu t}}\leqslant Ce^{\nu t}$ for some $\nu>0$, where the constant  $0<C<1$ does not depend on $t$.
\end{thm}

\begin{thm}\label{cor3.2.3con}
Let $n_0>n_1$, $n_0\geqslant 2$ and there exists a $j_0\in\{0,1,\ldots,n_0-2\}$ such that $\mathbb{K}_{j_0}=\emptyset$ and $\mathbb{K}_{j_0+1}\neq\emptyset$, and $h\in C[0,T]$. Then the initial value problem \eqref{eq1constant} and \eqref{eq2constant} has a unique solution $x\in C^{n_0-1,\beta_0}[0,T]$, which is given by
\begin{align*}
x(t)=&\sum_{j=0}^{n_0-1}c_j x_j(t)+\int_0^t \phi^{\prime}(s)(\phi(t)-\phi(s))^{\beta_0-1}\times \\
&\times E_{(\beta_0-\beta_1,\ldots,\beta_0-\beta_m),\beta_0}(-\lambda_1 (\phi(t)-\phi(s))^{\beta_0-\beta_1},\cdots,-\lambda_m (\phi(t)-\phi(s))^{\beta_0-\beta_m})h(s)ds,
\end{align*}
where 
\[x_j(t)=\Psi_j(t),\quad j=0,1,\ldots,j_0,\]
and 
\begin{align*}
x_j(t)=&\Psi_j(t)+\sum_{i=\varkappa_j}^m \lambda_i (\phi(t)-\phi(0))^{j+\beta_0-\beta_i}\times \\
&\times E_{(\beta_0-\beta_1,\ldots,\beta_0-\beta_m),j+1+\beta_0-\beta_i}(\lambda_1 (\phi(t)-\phi(0))^{\beta_0-\beta_1},\cdots,\lambda_m (\phi(t)-\phi(0))^{\beta_0-\beta_m}),
\end{align*}
for $j=j_0+1,\ldots,n_1-1$, where $\varkappa_j=\displaystyle{\min\{\mathbb{K}_j\}}$, and 
\begin{align*}
x_j(t)=&\Psi_j(t)+\sum_{i=1}^m \lambda_i (\phi(t)-\phi(0))^{j+\beta_0-\beta_i}\times \\
&\times E_{(\beta_0-\beta_1,\ldots,\beta_0-\beta_m),j+1+\beta_0-\beta_i}(\lambda_1 (\phi(t)-\phi(0))^{\beta_0-\beta_1},\cdots,\lambda_m (\phi(t)-\phi(0))^{\beta_0-\beta_m}),
\end{align*}
for $j=n_1,\ldots,n_0-1$, with $\Psi_j(t)=\frac{(\phi(t)-\phi(0))^{j}}{\Gamma(j+1)}$ and whenever $\displaystyle{\sum_{i=1}^m |\lambda_i|I_{0+}^{\beta_0-\beta_i,\phi}e^{\nu t}}\leqslant Ce^{\nu t}$ for some $\nu>0$, where the constant  $0<C<1$ does not depend on $t$.
\end{thm}

\begin{thm}\label{cor3.2.4con}
Let $n_0=n_1$, $\mathbb{K}_{n_0-1}=\emptyset$ and $h\in C[0,T]$. Then the initial value problem \eqref{eq1constant} and \eqref{eq2constant} has a unique solution $x\in C^{n_0-1,\beta_0}[0,T]$, which is given by
\begin{align*}
x(t)=&\sum_{j=0}^{n_0-1}c_j \Psi_j(t)+\int_0^t \phi^{\prime}(s)(\phi(t)-\phi(s))^{\beta_0-1}\times \\
&\times E_{(\beta_0-\beta_1,\ldots,\beta_0-\beta_m),\beta_0}(-\lambda_1 (\phi(t)-\phi(s))^{\beta_0-\beta_1},\cdots,-\lambda_m (\phi(t)-\phi(s))^{\beta_0-\beta_m})h(s)ds,
\end{align*}
where $\Psi_j(t)=\frac{(\phi(t)-\phi(0))^{j}}{\Gamma(j+1)}$ and whenever $\displaystyle{\sum_{i=1}^m |\lambda_i|I_{0+}^{\beta_0-\beta_i,\phi}e^{\nu t}}\leqslant Ce^{\nu t}$ for some $\nu>0$, where the constant  $0<C<1$ does not depend on $t$.
\end{thm}
\subsection{The case $\phi(t)=t$} The following corollaries follow by taking $\phi(t)=t$ in the above results. We formulate them for the convenience of the reader. We then consider
\begin{equation}\label{eq1constantcoro}
^{C}D_{0+}^{\beta_0}x(t)+\sum_{i=1}^{m}\lambda_i\,^{C}D_{0+}^{\beta_i}x(t)=h(t),\quad t\in[0,T],\; m\in\mathbb{N},
\end{equation}
under the initial conditions 
\begin{equation}\label{eq2constantcoro}
\left(\frac{d}{dt}\right)^k x(t)|_{_{t=+0}}=c_k\in\mathbb{R},\quad k=0,1,\ldots,n_0-1, 
\end{equation}
or 
\begin{equation}\label{eq4constantcoro}
\left(\frac{d}{dt}\right)^k x(t)|_{_{t=+0}}=0,\quad k=0,1,\ldots,n_0-1,
\end{equation}
where $\beta_i\in\mathbb{C}$, $\rm{Re}(\beta_{\it i})>0$, $i=0,1,\ldots,m-1$, $\rm{Re}(\beta_0)>\rm{Re}(\beta_1)>\ldots>\rm{Re}(\beta_{\it m})\geqslant0$ (If $\rm{Re}(\beta_{\it m})=0$, then we assume $\rm{Im}(\beta_{\it m})=0$ as well) and $n_i$ are non-negative integers satisfying $n_i-1<\rm{Re}(\beta_{\it i})<{\it n_i},$ $n_i=\lfloor \rm{Re}\beta_{\it i}\rfloor+1$ (or $n_i=-\lfloor-\rm{Re}(\beta_{\it i})\rfloor$), $i=0,1,\ldots,m$. Everywhere below we assume that $\varkappa_j=\displaystyle{\min\{\mathbb{K}_j\}}$.
\begin{cor}
Let $n_0>n_1$, $\beta_{m}=0$ and $h\in C[0,T]$.  Then the initial value problem \eqref{eq1constantcoro} and \eqref{eq2constantcoro} has a unique solution $x\in C^{n_0-1,\beta_0}[0,T]$ and it is given by
\begin{align*}
x(t)=&\sum_{j=0}^{n_0-1}c_j x_j(t)+\int_0^t (t-s)^{\beta_0-1}\times \\
&\times E_{(\beta_0-\beta_1,\ldots,\beta_0-\beta_m),\beta_0}(-\lambda_1 (t-s)^{\beta_0-\beta_1},\cdots,-\lambda_m (t-s)^{\beta_0-\beta_m})h(s)ds,
\end{align*}
where 
\begin{align*}
x_j(t)=&\frac{t^j}{\Gamma(j+1)}+\sum_{i=\varkappa_j}^m \lambda_i t^{j+\beta_0-\beta_i}E_{(\beta_0-\beta_1,\ldots,\beta_0-\beta_m),j+1+\beta_0-\beta_i}(\lambda_1 t^{\beta_0-\beta_1},\cdots,\lambda_m t^{\beta_0-\beta_m}),
\end{align*}
for  $j=0,1,\ldots,n_1-1$, and 
\begin{align*}
x_j(t)=&\frac{t^j}{\Gamma(j+1)}+\sum_{i=1}^m \lambda_i t^{j+\beta_0-\beta_i}E_{(\beta_0-\beta_1,\ldots,\beta_0-\beta_m),j+1+\beta_0-\beta_i}(\lambda_1 t^{\beta_0-\beta_1},\cdots,\lambda_m t^{\beta_0-\beta_m}),
\end{align*}
for $j=n_1,n_1+1,\ldots,n_0-1$.
\end{cor}    

\begin{cor}
Let $h\in C[0,T]$. Then the initial value problem \eqref{eq1constantcoro} and \eqref{eq4constantcoro} has a unique solution $x\in C^{n_0-1,\beta_0}[0,T]$ and it is given by the formula
\begin{align*}
x(t)&=\int_0^t (t-s)^{\beta_0-1}E_{(\beta_0-\beta_1,\ldots,\beta_0-\beta_m),\beta_0}(-\lambda_1 (t-s)^{\beta_0-\beta_1},\cdots,-\lambda_m (t-s)^{\beta_0-\beta_m})h(s)ds.
\end{align*}
\end{cor}

\begin{cor}
Let $n_0=n_1$, $\beta_{m}=0$ and $h\in C[0,T]$.  Then the initial value problem \eqref{eq1constantcoro} and \eqref{eq2constantcoro} has a unique solution $x\in C^{n_0-1,\beta_0}[0,T]$, which is given by
\begin{align*}
x(t)=&\sum_{j=0}^{n_0-1}c_j x_j(t)+\int_0^t (t-s)^{\beta_0-1}\times \\
&\times E_{(\beta_0-\beta_1,\ldots,\beta_0-\beta_m),\beta_0}(-\lambda_1 (t-s)^{\beta_0-\beta_1},\cdots,-\lambda_m (t-s)^{\beta_0-\beta_m})h(s)ds,
\end{align*}
where 
\begin{align*}
x_j(t)=&\frac{t^j}{\Gamma(j+1)}+\sum_{i=\varkappa_j}^m \lambda_i t^{j+\beta_0-\beta_i}E_{(\beta_0-\beta_1,\ldots,\beta_0-\beta_m),j+1+\beta_0-\beta_i}(\lambda_1 t^{\beta_0-\beta_1},\cdots,\lambda_m t^{\beta_0-\beta_m}),
\end{align*}
for  $j=0,1,\ldots,n_0-1$.
\end{cor}

\begin{cor}
Let $n_0=n_1$, $n_0\geqslant 2$ and there exists a $j_0\in\{0,1,\ldots,n_0-2\}$ such that $\mathbb{K}_{j_0}=\emptyset$ and $\mathbb{K}_{j_0+1}\neq\emptyset$, and $h\in C[0,T]$. Then the initial value problem \eqref{eq1constantcoro} and \eqref{eq2constantcoro} has a unique solution $x\in C^{n_0-1,\beta}[0,T]$, which is given by
\begin{align*}
x(t)=&\sum_{j=0}^{n_0-1}c_j x_j(t)+\int_0^t (t-s)^{\beta_0-1}\times \\
&\times E_{(\beta_0-\beta_1,\ldots,\beta_0-\beta_m),\beta_0}(-\lambda_1 (t-s)^{\beta_0-\beta_1},\cdots,-\lambda_m (t-s)^{\beta_0-\beta_m})h(s)ds,
\end{align*}
where 
\[x_j(t)=\Psi_j(t),\quad j=0,1,\ldots,j_0,\]
and 
\begin{align*}
x_j(t)=&\frac{t^j}{\Gamma(j+1)}+\sum_{i=\varkappa_j}^m \lambda_i t^{j+\beta_0-\beta_i}E_{(\beta_0-\beta_1,\ldots,\beta_0-\beta_m),j+1+\beta_0-\beta_i}(\lambda_1 t^{\beta_0-\beta_1},\cdots,\lambda_m t^{\beta_0-\beta_m}),
\end{align*}
for $j=j_0+1,\ldots,n_0-1$.
\end{cor}

\begin{cor}
Let $n_0>n_1$, $n_0\geqslant 2$ and there exists a $j_0\in\{0,1,\ldots,n_0-2\}$ such that $\mathbb{K}_{j_0}=\emptyset$ and $\mathbb{K}_{j_0+1}\neq\emptyset$, and $h\in C[0,T]$. Then the initial value problem \eqref{eq1constantcoro} and \eqref{eq2constantcoro} has a unique solution $x\in C^{n_0-1,\beta_0}[0,T]$, which is given by
\begin{align*}
x(t)=&\sum_{j=0}^{n_0-1}c_j x_j(t)+\int_0^t (t-s)^{\beta_0-1}\times \\
&\times E_{(\beta_0-\beta_1,\ldots,\beta_0-\beta_m),\beta_0}(-\lambda_1 (t-s)^{\beta_0-\beta_1},\cdots,-\lambda_m (t-s)^{\beta_0-\beta_m})h(s)ds,
\end{align*}
where 
\[x_j(t)=\frac{t^j}{\Gamma(j+1)},\quad j=0,1,\ldots,j_0,\]
and 
\begin{align*}
x_j(t)=&\frac{t^j}{\Gamma(j+1)}+\sum_{i=\varkappa_j}^m \lambda_i t^{j+\beta_0-\beta_i}E_{(\beta_0-\beta_1,\ldots,\beta_0-\beta_m),j+1+\beta_0-\beta_i}(\lambda_1 t^{\beta_0-\beta_1},\cdots,\lambda_m t^{\beta_0-\beta_m}),
\end{align*}
for $j=j_0+1,\ldots,n_1-1$, and 
\begin{align*}
x_j(t)=&\frac{t^j}{\Gamma(j+1)}+\sum_{i=1}^m \lambda_i t^{j+\beta_0-\beta_i}E_{(\beta_0-\beta_1,\ldots,\beta_0-\beta_m),j+1+\beta_0-\beta_i}(\lambda_1 t^{\beta_0-\beta_1},\cdots,\lambda_m t^{\beta_0-\beta_m}),
\end{align*}
for $j=n_1,\ldots,n_0-1$.
\end{cor}

\begin{cor}
Let $n_0=n_1$, $\mathbb{K}_{n_0-1}=\emptyset$ and $h\in C[0,T]$. Then the initial value problem \eqref{eq1constantcoro} and \eqref{eq2constantcoro} has a unique solution $x\in C^{n_0-1,\beta_0}[0,T]$, which is given by
\begin{align*}
x(t)=&\sum_{j=0}^{n_0-1}c_j \frac{t^j}{\Gamma(j+1)}+\int_0^t (t-s)^{\beta_0-1}\times \\
&\times E_{(\beta_0-\beta_1,\ldots,\beta_0-\beta_m),\beta_0}(-\lambda_1 (t-s)^{\beta_0-\beta_1},\cdots,-\lambda_m (t-s)^{\beta_0-\beta_m})h(s)ds.
\end{align*}
\end{cor}

\section{Concluding remarks}\label{remarks}

In a particular case our results presented in Section \ref{addition} imply \cite[Theorem 4.1]{luchko}. That is, for example, in the case of constant coefficients of real orders and $\phi(x)=x$, Theorem \ref{cor3.2.1constant} coincides with \cite[Theorem 4.1]{luchko}. The method developed in \cite{luchko} used a very different approach from the current paper. So, it is important to verify the consistency of these results. Below we recall \cite[Theorem 4.1]{luchko} to compare easily the coincidence of the results. Firstly, we introduce some notation. For any $\beta\in\mathbb{R}$ we define
\[C_{\beta}:=\left\{f:\mathbb{R}^+\to\mathbb{C}:\exists q\in\mathbb{R}, q>\beta, \text{such that} \; f(x)=x^qf_1(x),\, f_1\in C[0,+\infty)\right\}\]
and for $n\in\mathbb{N}_0$ we have
\[f\in C_{\beta}^n\quad\text{if and only if}\quad f^{(n)}\in C_{\beta}.\]
Here we assume that $C_\beta^0\equiv C_\beta$. Now let us recall \cite[Theorem 4.1]{luchko}. 
\begin{thm}\cite[Theorem 4.1]{luchko} \label{thmluchko}
Let $\beta_0>\beta_1>\cdots>\beta_m\geqslant0$, $n_i-1<\beta_i\leqslant n_i$, $n_i\in\mathbb{N}_0=\mathbb{N}\cup\{0\}$, $\lambda_i\in\mathbb{R}$, $i=1,\ldots,m.$ Consider the following initial value problem 
\begin{equation}\label{eq1luchko}
\begin{split}
^{C}D_{0+}^{\beta_0}x(t)&+\sum_{i=1}^{m}\lambda_i\,^{C}D_{0+}^{\beta_i}x(t)=h(t),\quad t\in[0,T], \\
\left(\frac{d}{dt}\right)^k x(t)|_{_{t=+0}}&=c_k\in\mathbb{R},\quad k=0,1,\ldots,n_0-1,\,\,n_0-1<\beta_0\leqslant n_0,
\end{split} 
\end{equation}
where $h$ is assumed to be in $C_{-1}$ if $\beta_0\in\mathbb{N}$ or in $C_{-1}^1$ if $\beta_0\notin\mathbb{N}$. Then the problem \eqref{eq1luchko} has a unique solution in the space $C_{-1}^{n_0}$ of the form
\begin{equation}\label{solgluchko}
x(t)=\int_0^t s^{\beta_0-1}E_{(\beta_0-\beta_1,\ldots,\beta_0-\beta_m),\beta_0}(\lambda_1 s^{\beta_0-\beta_1},\cdots,\lambda_m s^{\beta_0-\beta_m})h(t-s)ds+\sum_{k=0}^{n_0-1}c_k x_k(t),
\end{equation}
where 
\begin{equation}\label{solpluchko}
x_k(t)=\frac{t^{k}}{k!}+\sum_{i=l_k+1}^m \lambda_i t^{k+\beta_0-\beta_i}E_{(\beta_0-\beta_1,\ldots,\beta_0-\beta_m),k+1+\beta_0-\beta_i}(\lambda_1 s^{\beta_0-\beta_1},\cdots,\lambda_m s^{\beta_0-\beta_m})
\end{equation}
and $x_k^{(l)}(0)=\delta_{kl}$, $k,l=0,\ldots,n_0-1$. The natural numbers $l_k$ for $k=0,\ldots,n_0-1$ are determined from the conditions $n_{l_k}\geqslant k+1$ and $n_{l_k +1}\leqslant k$. In the case $n_i\leqslant k$ for $i=0,\ldots,n_0-1$, we set $l_k=0$, while $l_k=m$ for $n_i\geqslant k+1$ for $i=0,\ldots,n_0-1$.
\end{thm}  

In all theorems of this paper (such as Theorem \ref{thm3.1}, Theorem \ref{cor3.2.1constant} and Theorem \ref{thmluchko}) analytical solutions are presented and these solutions are indeed unique ``classical'' ones (with respect to the function $h$).  It is clear that the function $h$ in Theorem \ref{thmluchko} must be in $C_{-1}$ or $C_{-1}^1$ that it is less restrictive than in our case that can be assumed to be in $C$ since $C\subset C_{-1}$. However, of course, Theorem \ref{thm3.1}
covers much more general case and we believe one can weaken the assumption on the function $h$ in special cases. 

Finally, we should mention that a different approach to find an  analytic solution of FDEs  with variable coefficients was given in \cite{AML}. The authors studied the case of homogeneous FDEs with variable coefficients for Riemann-Liouville fractional derivatives with boundary values involving fractional derivatives as well. In turn, those results in \cite{AML} can be used to extend our techniques to not only initial (Cauchy) problems, but also initial boundary value problems with the Riemann-Liouville fractional derivatives.  

\section{Acknowledgements}

This research is funded by the Science Committee of the Ministry of Education and Science of the Republic of Kazakhstan (Grant No. AP09058317). The first and third authors were supported by the Nazarbayev University Program 091019CRP2120. Joel E. Restrepo also thanks to Colciencias and Universidad de Antioquia (Convocatoria 848 - Programa de estancias postdoctorales 2019) for their support. The second author was supported by the FWO Odysseus 1 grant G.0H94.18N: Analysis and Partial Differential Equations and by the EPSRC Grant EP/R003025.


\bibliography{mybibfile}

\end{document}